\documentclass[12pt]{amsart}

\usepackage[matrix,arrow,curve,frame]{xy}
\usepackage{amsmath,amsthm,amssymb,enumerate}
\usepackage{latexsym}
\usepackage{amscd}
\usepackage[colorlinks=false]{hyperref}
\usepackage[mathscr]{eucal}

\newtheorem{thm}[subsection]{Theorem}

\newtheorem{cor}[subsection]{Corollary}
\newtheorem{lemma}[subsection]{Lemma}

\newtheorem{remark}[subsection]{Remark}

\theoremstyle{definition}

\numberwithin{equation}{section}

\def\C{\mathbb {C}}

\def\Z{\mathbb Z}

\def\phi{{\varphi}}

\def\ra{\rightarrow}
\def\bra{\langle}
\def\ket{\rangle}

\def\cA{{\mathcal A}}
\def\cB{{\mathcal B}}
\def\cC{{\mathcal C}}
\def\cD{{\mathcal D}}
\def\cE{{\mathcal E}}

\def\cH{{\mathcal H}}
\def\cI{{\mathcal I}}
\def\cJ{{\mathcal J}}

\def\cV{{\mathcal V}}
\def\cW{{\mathcal W}}

\def\gg{{\mathfrak g}}

\def\gl{{\mathfrak l}}

\def\gs{{\mathfrak s}}

\newcommand{\mf}{\mathfrak}
\newcommand{\on}{\operatorname}
\newcommand{\lam}{\lambda}
\newcommand{\Lam}{\Lambda}
\newcommand{\fing}{\mf{g}}

\newcommand{\isomap}{{\;\stackrel{_\sim}{\to}\;}}

\DeclareMathOperator{\ch}{ch}

\newfont{\german}{eufm10}

\begin{document}
\pagestyle{plain}

\title
{Cosets of Bershadsky-Polyakov algebras and rational $\cW$-algebras of type $A$}

\author{Tomoyuki Arakawa, Thomas Creutzig and Andrew R. Linshaw}
\address{Research Institute for Mathematical Sciences, Kyoto University}
\email{arakawa@kurims.kyoto-u.ac.jp}
\address{Department of Mathematics, University of Alberta}
\email{creutzig@ualberta.ca}
\address{Department of Mathematics, University of Denver}
\email{andrew.linshaw@du.edu}
\thanks{This work was partially supported by JSPS KAKENHI Grants
(\#25287004 and \#26610006 to T. Arakawa), an NSERC Discovery Grant (\#RES0019997 to T. Creutzig), and a grant from the Simons Foundation (\#318755 to A. Linshaw)}

\begin{abstract} 
The Bershadsky-Polyakov algebra is the $\cW$-algebra associated to $\gs\gl_3$ with its minimal nilpotent element $f_{\theta}$. For notational convenience we define $\cW^{\ell} = \cW^{\ell - 3/2} (\gs\gl_3, f_{\theta})$. The simple quotient of $\cW^{\ell}$ is denoted by $\cW_{\ell}$, and for $\ell$ a positive integer, $\cW_{\ell}$ is known to be $C_2$-cofinite and rational. We prove that for all positive integers $\ell$, $\cW_{\ell}$ contains a rank one lattice vertex algebra $V_L$, and that the coset $\cC_{\ell} = \text{Com}(V_L, \cW_{\ell})$ is isomorphic to the principal, rational $\cW(\gs\gl_{2\ell})$-algebra at level $(2\ell +3)/(2\ell +1) -2\ell$. This was conjectured in the physics literature over 20 years ago. As a byproduct, we construct a new family of rational, $C_2$-cofinite vertex superalgebras from $\cW_{\ell}$. \end{abstract}

\maketitle
\section{Introduction}
A longstanding problem in vertex algebra theory is to classify rational and $C_2$-cofinite vertex algebras. Well-known examples include those associated to even positive definite lattices \cite{D}, and simple affine vertex algebras at positive integral level \cite{FZ}. More recently, Kac and Wakimoto conjectured the $C_2$-cofiniteness and rationality of many quantum Hamiltonian reductions based on the modularity of their characters \cite{KW}, and these conjectures were proven for the minimal series principal $\cW$-algebras in \cite{ArIV, ArV}. New examples beyond this conjecture and still associated to quantum Hamiltonian reductions have also been found \cite{Ka, ArM}. The classification problem seems to be out of reach at present, but it is still of great interest to find new examples and constructions of such vertex algebras.

There are some standard ways to construct new vertex algebras from old ones. First, the {\it orbifold construction} begins with a vertex algebra $\cV$ and a group of automorphisms $G$, and considers the invariant subalgebra $\cV^G$ and its extensions. The Moonshine vertex algebra is constructed in this way; its full automorphism group is the Monster, which is the largest of the sporadic simple finite groups, and its graded character is $j(\tau) - 744$ where $j(\tau)$ is the modular invariant $j$-function \cite{FLM}. It is widely believed that if $\cV$ is $C_2$-cofinite and rational, and $G$ is a finite group, $\cV^G$ will inherit these properties. In the case where $G$ is cyclic, the $C_2$-cofiniteness of $\cV^G$ was proven by Miyamoto in \cite{MiI}, and the rationality was established by Carnahan and Miyamoto in \cite{CM}.

Similarly, the {\it coset construction} associates to a vertex algebra $\cV$ and a subalgebra $\cA$, the subalgebra $\text{Com}(\cA,\cV) \subset \cV$ which commutes with $\cA$. This was introduced by Frenkel and Zhu in \cite{FZ}, generalizing earlier constructions in representation theory \cite{KP} and physics \cite{GKO}, where it was used to construct the unitary discrete series representations of the Virasoro algebra. It is believed that if both $\cA$ and $\cV$ are $C_2$-cofinite and rational, these properties will be inherited by $\text{Com}(\cA,\cV)$. In the case where $\cV$ is a rational affine vertex algebra and $\cA$ is a lattice vertex algebra, this has been proven recently in \cite{DR}, building on \cite{ALY, DLY, DLWY, DWI, DWII, DWIII}.

In this paper, we consider cosets of the Bershadsky-Polyakov algebra. This algebra appeared originally in the physics literature \cite{Ber, Pol}, and is the $\cW$-algebra corresponding to the quantum Hamiltonian reduction of $\gs\gl_3$ with its minimal nilpotent element $f_{\theta}$. For notational convenience, we introduce a shift of level and define $\cW^{\ell}$ to be the universal Bershadsky-Polyakov algebra $\cW^{\ell - 3/2}(\gs\gl_3, f_{\theta})$, for $\ell \neq -3/2$. It is freely generated by fields $J, T,G^+, G^-$ of weights $1, 2, 3/2, 3/2$, and we denote its simple quotient by $\cW_{\ell}$. For $\ell$ a positive integer, the $C_2$-cofiniteness and rationality of $\cW_{\ell}$ were proven in \cite{ArII}. Let $$\cC^{\ell} = \text{Com}(\cH, \cW^{\ell}),$$ where $\cH$ is the Heisenberg algebra generated by $J$. It was conjectured over 20 years ago in the physics literature \cite{B-HII} that $\cC^{\ell}$ should be of type $\cW(2,3,4,5,6,7)$ for generic values of $\ell$. In other words, a minimal strong generating set should consist of one field in each of these weights. In \cite{B-HII} it was also conjectured that when $\ell$ is a positive integer, the simple coset $$\cC_{\ell} = \text{Com}(\cH, \cW_{\ell})$$ should be the principal $\cW(\gs\gl_{2\ell})$-algebra at level $(2\ell +3)/(2\ell +1) -2\ell$, which is rational by \cite{ArV}. So far, this has been proven only very recently by Kawasetsu \cite{Ka} in the case $\ell = 1$. These conjectures imply a {\it uniform truncation} property of this family of simple $\cW(\gs\gl_{2\ell})$-algebras for $\ell>3$; they are expected to be of type $\cW(2,3,4,5,6,7)$, even though the universal $\cW(\gs\gl_{2\ell})$-algebra is of type $\cW(2,3,\dots, 2\ell)$. This can only happen if there exist decoupling relations expressing the generators of weights $8,9,\dots, 2\ell$ as normally ordered polynomials in the ones up to weight $7$. The existence of a singular vector of weight $8$ in the universal algebra has been known for many years \cite{B-HI}, and provides further evidence for these conjecture.

In this paper we shall prove all of these conjectures. We begin by studying the $U(1)$-orbifold $(\cW^{\ell})^{U(1)}$ in Section \ref{sect:structureorbifold}. Typically, an orbifold of a vertex algebra $\cA^{\ell}$ with structure constants depending continuously on a parameter ${\ell}$ will have a minimal strong generating set that works for generic values of ${\ell}$. This was established in \cite{CLI} when $\cA^{\ell}$ is a tensor product of affine and free field algebras, and a similar approach works for $\cW^{\ell}$. We show that $(\cW^{\ell})^{U(1)}$ is of type $\cW(1,2,3,4,5,6,7)$ for generic values of ${\ell}$, and we give a complete description of the {\it nongeneric} set, where additional generators are needed; it consists only of $\{0,1/2\}$. In general, it is an important problem to characterize the nongeneric set for such orbifolds, and we expect that our methods will be applicable in a wider setting. In Section \ref{sect:structuregeneric}, using the description of $(\cW^{\ell})^{U(1)}$, we obtain

\begin{thm} $\cC^{\ell} = \text{Com}(\cH, \cW^{\ell})$ is of type $\cW(2,3,4,5,6,7)$ for all values of $\ell$ except for $0$ and $1/2$. Here $\cH$ is the Heisenberg algebra generated by $J$. \end{thm}

In Section \ref{sect:structuresimple}, we show that when $\ell$ is a positive integer, the Heisenberg algebra $\cH \subset \cW_{\ell}$ is actually part of a rank one lattice vertex algebra $V_L$ for $L = \sqrt{6\ell} \mathbb{Z}$. Moreover, $V_L$ and $\cC_{\ell}$ form a {\it Howe pair}, i.e., a pair of mutual commutants, inside $\cW_{\ell}$. By a theorem of Miyamoto, this implies the $C_2$-cofiniteness of $\cC_{\ell}$. In Section \ref{sect:WtypeA}, we consider a certain simple current extension of $\cW(\gs\gl_{2\ell}) \otimes V_L$ where $\cW(\gs\gl_{2\ell})$ has level $(2\ell +3)/(2\ell +1) -2\ell$, and show that it is isomorphic to $\cW_{\ell}$. This implies our main result:

\begin{thm} For all positive integers $\ell$, $\cC_{\ell}$ is isomorphic to the principal, rational $\cW(\gs\gl_{2\ell})$-algebra at level $(2\ell +3)/(2\ell +1) -2\ell$.
\end{thm} 

We obtain the uniform truncation property of this family of $\cW(\gs\gl_{2\ell})$-algebras, and also a new proof of the $C_2$-cofiniteness and rationality of $\cW_{\ell}$, since it is a simple current extension of a rational, $C_2$-cofinite vertex algebra. The appearance of principal, rational $\cW$-algebras as cosets of nonprincipal ones is a remarkable coincidence, and it indicates that the principal rational $\cW$-algebras may be important building blocks for more general rational $\cW$-algebras. In Section \ref{sect:character}, we obtain some nontrivial character identities as a consequence of the above isomorphism. Finally, in Section \ref{sect:newvsa} we consider cosets of the Heisenberg algebra inside $\cW^{\ell} \otimes \cE$ where $\cE$ denotes the rank one $bc$-system. We obtain a new deformable vertex superalgebra with even generators in weights $1,2,3$ and two odd generators in weight $2$, whose simple quotients for positive integers $\ell$ form a new series of rational, $C_2$-cofinite vertex superalgebras with integer conformal weights.

Coset vertex algebras have various applications in physics. One of the most interesting recent developments in the AdS/CFT correspondence is the relation between two-dimensional coset conformal field theories and higher spin gravity on three dimensional Anti-de-Sitter space. The main examples are the $\cW_N$-minimal models \cite{GG} and their super conformal analogues \cite{CHR}. Recall that the Bershadsky-Polyakov algebra coincides with the case $n=3$ in the family of $\cW^{(2)}_n$-algebras constructed by Feigin and Semikhatov \cite{FS}, which were studied this context in \cite{A-R}. In that article it was  observed that the $\cW^{(2)}_4$-algebra at level $-7/3$ is just a rank one lattice vertex algebra and hence rational. We expect that for all $n\geq 4$, there is a similar series of rational $\cW^{(2)}_n$-algebras which contain a rank one lattice vertex algebra $V_L$, and that the coset of $V_L$ will give rise to another series of principal, rational $\cW$-algebras of type $A$ with a uniform truncation property. In fact, coincidences among cosets of simple $\cW$-algebras by the affine subalgebra at special levels is a much more general phenomenon that is not limited to the rational levels. Other examples of this phenomenon involving cosets of minimal $\cW$-algebras can be found in \cite{ACKL}.

Another reason to study cosets $\cC_{\ell}$ at irrational levels is that they have some interesting connections to {\it logarithmic conformal field theories}, which are associated to irrational $C_2$-cofinite vertex algebras whose representation category is not semisimple. The best studied examples are the so-called $\cW(p)$-triplet algebras, see e.g. \cite{AM, TW}. Interestingly, $\cW(p)$ is related to the coset $\text{Com}(\cH, \cW^{(2)}_{p-1})$ where $\cH$ is a Heisenberg algebra. In particular, $\cW(4)$ is a simple current extension of our coset $\cC_{\ell}$ with $\ell = -3/4$ \cite{CRW}.

\section{Vertex algebras}
In this section, we define vertex algebras, which have been discussed from various points of view in the literature  (see for example \cite{Bor,FLM,K,FBZ}). We will follow the formalism developed in \cite{LZ} and partly in \cite{LiI}. Let $V=V_0\oplus V_1$ be a super vector space over $\mathbb{C}$, and let $z,w$ be formal variables. By $\text{QO}(V)$, we mean the space of linear maps $$V\ra V((z))=\{\sum_{n\in\mathbb{Z}} v(n) z^{-n-1}|
v(n)\in V,\ v(n)=0\ \text{for} \ n>\!\!>0 \}.$$ Each element $a\in \text{QO}(V)$ can be represented as a power series
$$a=a(z)=\sum_{n\in\mathbb{Z}}a(n)z^{-n-1}\in \text{End}(V)[[z,z^{-1}]].$$ We assume that $a=a_0+a_1$ where $a_i:V_j\ra V_{i+j}((z))$ for $i,j\in\mathbb{Z}/2\mathbb{Z}$, and we write $|a_i| = i$.

For each $n \in \mathbb{Z}$, there is a nonassociative bilinear operation $\circ_n$ on $\text{QO}(V)$, defined on homogeneous elements $a$ and $b$ by
$$ a(w)\circ_n b(w)=\text{Res}_z a(z)b(w)\ \iota_{|z|>|w|}(z-w)^n- (-1)^{|a||b|}\text{Res}_z b(w)a(z)\ \iota_{|w|>|z|}(z-w)^n.$$
Here $\iota_{|z|>|w|}f(z,w)\in\mathbb{C}[[z,z^{-1},w,w^{-1}]]$ denotes the power series expansion of a rational function $f$ in the region $|z|>|w|$. For $a,b\in \text{QO}(V)$, we have the following identity of power series known as the {\it operator product expansion} (OPE) formula.
 \begin{equation}\label{opeform} a(z)b(w)=\sum_{n\geq 0}a(w)\circ_n
b(w)\ (z-w)^{-n-1}+:a(z)b(w):. \end{equation}
Here $:a(z)b(w):\ =a(z)_-b(w)\ +\ (-1)^{|a||b|} b(w)a(z)_+$, where $a(z)_-=\sum_{n<0}a(n)z^{-n-1}$ and $a(z)_+=\sum_{n\geq 0}a(n)z^{-n-1}$. Often, \eqref{opeform} is written as
$$a(z)b(w)\sim\sum_{n\geq 0}a(w)\circ_n b(w)\ (z-w)^{-n-1},$$ where $\sim$ means equal modulo the term $:a(z)b(w):$, which is regular at $z=w$. 

Note that $:a(z)b(z):$ is a well-defined element of $\text{QO}(V)$. It is called the {\it Wick product} or {\it normally ordered product} of $a$ and $b$, and it
coincides with $a(z)\circ_{-1}b(z)$. For $n\geq 1$ we have
$$ n!\ a(z)\circ_{-n-1}b(z)=\ :(\partial^n a(z))b(z):,\qquad \partial = \frac{d}{dz}.$$
For $a_1(z),\dots ,a_k(z)\in \text{QO}(V)$, the $k$-fold iterated Wick product is defined inductively by
\begin{equation}\label{iteratedwick} :a_1(z)a_2(z)\cdots a_k(z):\ =\ :a_1(z)b(z):,\qquad b(z)=\ :a_2(z)\cdots a_k(z):.\end{equation}
We often omit the formal variable $z$ when no confusion can arise.

A subspace $\cA\subset \text{QO}(V)$ containing $1$ which is closed under all products $\circ_n$ will be called a {\it quantum operator algebra} (QOA). We say that $a,b\in \text{QO}(V)$ are {\it local} if $$(z-w)^N [a(z),b(w)]=0$$ for some $N\geq 0$. Here $[,]$ denotes the super bracket. This condition implies that $a\circ_n b = 0$ for $n\geq N$, so (\ref{opeform}) becomes a finite sum. Finally, a {\it vertex algebra} will be a QOA whose elements are pairwise local. This notion is well known to be equivalent to the notion of a vertex algebra in the sense of \cite{FLM}. 

A vertex algebra $\cA$ is said to be {\it generated} by a subset $S=\{a_i|\ i\in I\}$ if $\cA$ is spanned by words in the letters $a_i$, $\circ_n$, for $i\in I$ and $n\in\mathbb{Z}$. We say that $S$ {\it strongly generates} $\cA$ if $\cA$ is spanned by words in the letters $a_i$, $\circ_n$ for $n<0$. Equivalently, $\cA$ is spanned by $$\{ :\partial^{k_1} a_{i_1}\cdots \partial^{k_m} a_{i_m}:| \ i_1,\dots,i_m \in I,\ k_1,\dots,k_m \geq 0\}.$$ We say that $S$ {\it freely generates} $\cA$ if there are no nontrivial normally ordered polynomial relations among the generators and their derivatives. As a matter of notation, we say that a vertex algebra $\cA$ is of type $\cW(d_1,\dots, d_r)$ if it has a minimal strong generating set consisting of one field in each weight $d_1,\dots, d_r$.

Given a vertex algebra $\cV$ and a vertex subalgebra $\cA \subset \cV$, the {\it coset} or {\it commutant} of $\cA$ in $\cV$, denoted by $\text{Com}(\cA,\cV)$, is the subalgebra of elements $v\in\cV$ such that $[a(z),v(w)] = 0$ for all $a\in\cA$. Equivalently, $v\in \text{Com}(\cA,\cV)$ if and only if $a\circ_n v = 0$ for all $a\in\cA$ and $n\geq 0$.

\section{The Bershadsky-Polyakov algebra}
The universal Bershadsky-Polyakov algebra at level $k \neq -3$ is isomorphic to the $\cW$-algebra $\cW^k(\gs\gl_3, f_{\theta})$ associated to $\gs\gl_3$ with its minimal nilpotent element $f_{\theta}$ (\cite{KRW}). Alternatively, it arises as the algebra $\cW^{(2)}_3$ in the family $\cW^{(2)}_n$ constructed by Feigin and Semikhatov \cite{FS}. For notational convenience we introduce a shift of level and define $$\cW^{\ell} = \cW^{\ell-3/2}(\gs\gl_3, f_{\theta}).$$ The reason for this shift is that the values $\ell = 0,1,2,\dots$ now correspond to the special levels $p/2 -3$ for $p=3,5,7,\dots$, which were considered in \cite{ArII}. For $\ell \neq -3/2$, $\cW^{\ell}$ is freely generated by fields $J,T,G^+, G^-$ of weights $1, 2, 3/2, 3/2$, satisfying the following OPE relations. 

\begin{equation} \label{bp1} J(x) J(w) \sim \frac{2\ell}{3} (z-w)^{-2},\qquad G^{\pm}(z) G^{\pm}(w) \sim 0,\qquad J(z) G^{\pm}(w) \sim \pm G^{\pm}(w)(z-w)^{-1},\end{equation}
\begin{equation} \label{bp2} T(z) T(w) \sim - \frac{\ell(6\ell -7)}{2\ell +3} (z-w)^{-4} + 2 T(w)(z-w)^{-2} + \partial T(w)(z-w)^{-1},\end{equation}
\begin{equation} \label{bp3} T(z) G^{\pm}(w) \sim \frac{3}{2} G^{\pm}(w) (z-w)^{-2} + \partial G^{\pm}(w)(z-w)^{-1},\end{equation}
\begin{equation} \label{bp4} T(z) J(w) \sim J(w)(z-w)^{-2} + \partial J(w)(z-w)^{-1},\end{equation}
\begin{equation} \label{bp5} \begin{split} G^+(z) G^-(w) \sim \ell(2\ell-1)(z-w)^{-3} + \frac{3}{2}(2\ell-1) J(w) (z-w)^{-2} \\ + \bigg( 3:J(w)J(w): + \frac{3}{4}(2\ell-1) \partial J(w) - (\ell+\frac{3}{2})T(w)\bigg)(z-w)^{-1}. \end{split} \end{equation}
We denote by $\cW_{\ell}$ the simple quotient of $\cW^{\ell}$ by its maximal, proper ideal graded by conformal weight. For $\ell = 0$, $\cW_{0} \cong \mathbb{C}$, and for $\ell$ a positive integer, the maximal proper ideal is generated by $(G^{\pm})^{2\ell+1}$ and $\cW_{\ell}$ is $C_2$-cofinite and rational \cite{ArII}.

\section{Weak increasing filtrations}
A {\it good increasing filtration} on a vertex algebra $\cA$ is a $\mathbb{Z}_{\geq 0}$-filtration
\begin{equation} \cA_{(0)}\subset\cA_{(1)}\subset\cA_{(2)}\subset \cdots,\qquad \cA = \bigcup_{d\geq 0}
\cA_{(d)}\end{equation} such that $\cA_{(0)} = \mathbb{C}$, and for all
$a\in \cA_{(r)}$, $b\in\cA_{(s)}$, we have
\begin{equation} \label{goodi} a\circ_n b\in  \bigg\{\begin{matrix}\cA_{(r+s)} & n<0 \\ \cA_{(r+s-1)} & 
n\geq 0 \end{matrix}\ . \end{equation}
We set $\cA_{(-1)} = \{0\}$, and we say that $a \in\cA_{(d)}\setminus \cA_{(d-1)}$ has degree $d$. Such filtrations were introduced in \cite{LiII}, and the key property is that the associated graded algebra $$\text{gr}(\cA) = \bigoplus_{d\geq 0}\cA_{(d)}/\cA_{(d-1)}$$ is a $\mathbb{Z}_{\geq 0}$-graded associative, (super)commutative algebra with a
unit $1$ under a product induced by the Wick product on $\cA$. The derivation $\partial$ on $\cA$ induces a derivation of degree zero on $\text{gr}(\cA)$, which we also denote by $\partial$. We call such rings {\it $\partial$-rings}, and we say that a $\partial$-ring $A$ is generated by a set $\{a_i|\ i\in I\}$ if $\{\partial^k a_i|\ i\in I, k\geq 0\}$ generates $A$ as a ring. 

For $r\geq 1$ we have the projection \begin{equation} \phi_r: \cA_{(r)} \ra \cA_{(r)}/\cA_{(r-1)}\subset \text{gr}(\cA).\end{equation} By Lemma 3.6 of \cite{LL}, if $\{a_i|\ i\in I\}$ is a set of generators for $\text{gr}(\cA)$ as a $\partial$-ring, where $a_i$ is homogeneous of degree $d_i$, and $\tilde{a}_i \in\cA_{(d_i)}$ are elements satisfying $\phi_{d_i}(\tilde{a}_i) = a_i$, then $\cA$ is strongly generated as a vertex algebra by $\{\tilde{a}_i|\ i\in I\}$.

In this paper, we will need filtrations with weaker properties. We define a {\it weak increasing filtration} on a vertex algebra $\cA$ to be a $\mathbb{Z}_{\geq 0}$-filtration
\begin{equation}\label{weak} \cA_{(0)}\subset\cA_{(1)}\subset\cA_{(2)}\subset \cdots,\qquad \cA = \bigcup_{d\geq 0}
\cA_{(d)}\end{equation} such that for $a\in \cA_{(r)}$, $b\in\cA_{(s)}$, we have
\begin{equation} a\circ_n b \in \cA_{(r+s)},\qquad n\in \mathbb{Z}.\end{equation}
This condition guarantees that $\text{gr}(\cA) = \bigoplus_{d\geq 0}\cA_{(d)}/\cA_{(d-1)}$ is a vertex algebra, but it is no longer abelian in general. Let  $\phi_r: \cA_{(r)} \ra \cA_{(r)}/\cA_{(r-1)}\subset \text{gr}(\cA)$ denote the projection as above. We have the following reconstruction property, whose proof is the same as the proof of Lemma 3.6 of \cite{LL}.
\begin{lemma}\label{reconlem}Let $\cA$ be a vertex algebra with a weak increasing filtration, and let $\{a_i|\ i\in I\}$ be a set of strong generators for $\text{gr}(\cA)$, where $a_i$ is homogeneous of degree $d_i$. If $\tilde{a}_i$ are elements of $\cA_{(d_i)}$ satisfying $\phi_{d_i}(\tilde{a}_i) = a_i$ for all $i\in I$, then $\cA$ is strongly generated as a vertex algebra by $\{\tilde{a}_i|\ i\in I\}$.\end{lemma}

We define an increasing filtration 
$$\cW^{\ell}_{(0)} \subset \cW^{\ell}_{(1)} \subset \cdots $$ on $\cW^{\ell}$ as follows: $\cW^{\ell}_{(-1)} = \{0\}$, and $\cW^{\ell}_{(r)}$ is spanned by iterated Wick products of the generators $J,T,G^{\pm}$ and their derivatives, such that at most $r$ of $G^{\pm}$ and their derivatives appear. It is clear from the defining OPE relations that this is a weak increasing filtration. Clearly $\cW^{\ell}_{(0)}$ has strong generators $J$ and $T$. Replacing $T$ by $T^{\cC} = T- \frac{4}{3\ell} :JJ:$, we see that $J(z) T^{\cC}(w) \sim 0$, so $\cW^{\ell}_{(0)}$ is the tensor product of the Heisenberg algebra with generator $J$ and the Virasoro algebra with generator $T^{\cC}$.

Note that the associated graded algebra $$\cV^{\ell} = \text{gr}(\cW^{\ell}) = \bigoplus_{d \geq 0} \cW^{\ell}_{(d)} / \cW^{\ell}_{(d-1)}$$ is freely generated by $J,T,G^{\pm}$. (By abuse of notation, we use the same symbols for the images of these generators in $\cV^{\ell}$). The OPE relations \eqref{bp1}-\eqref{bp4} still hold in $\cV^{\ell}$, but \eqref{bp5} is replaced with $G^+(z) G^-(w) \sim 0$. Finally, $\cV^{\ell}$ has a good increasing filtration $$\cV^{\ell}_{(0)} \subset \cV^{\ell}_{(1)} \subset \cdots,$$ where  $\cV^{\ell}_{(-1)} = \{0\}$, and $\cV^{\ell}_{(r)}$ is spanned by iterated Wick products of the generators $J,T,G^{\pm}$ and their derivatives, of length at most $r$. Then $\text{gr}(\cV^{\ell})$ is an abelian vertex algebra freely generated by $J,T,G^{\pm}$. In particular, \begin{equation} \text{gr}(\cV^{\ell})\cong \mathbb{C}[J,\partial J, \partial^2 J,\dots, T,\partial T, \partial^2 T,\dots, G^+, \partial G^+,\partial^2 G^+,\dots, G^-, \partial G^-,\partial^2 G^-,\dots].\end{equation}

\section{The $U(1)$ invariants in $\cW^{\ell}$} \label{sect:structureorbifold}
The action of the zero mode $J_0$ integrates to a $U(1)$-action on $\cW^{\ell}$, and the invariant subalgebra $(\cW^{\ell})^{U(1)}$ coincides with the kernel of $J_0$. Since $J,T$ lie in $(\cW^{\ell})^{U(1)}$ and $J_0(G^{\pm}) = \pm G^{\pm}$, it is immediate that $(\cW^{\ell})^{U(1)}$ is spanned by all normally ordered monomials of the form 
\begin{equation} \label{standardmonomial} :(\partial^{a_1} T) \cdots (\partial^{a_i} T) (\partial^{b_1} J )\cdots (\partial^{b_j} J) (\partial^{c_1} G^+) \cdots (\partial^{c_r} G^+)( \partial^{d_1} G^-) \cdots (\partial^{d_r} G^-):,\end{equation} where $r\geq 0$ and $a_1\geq \cdots \geq a_i \geq 0$, $b_1\geq \cdots \geq b_j \geq 0$, $c_1\geq \cdots \geq c_r \geq 0$, and $d_1\geq \cdots \geq d_r \geq 0$. We say that $\omega \in (\cW^{\ell})^{U(1)}$ is in {\it normal form} if it has been expressed as a linear combination of such monomials. Since $\cW^{\ell}$ is freely generated by $J,T,G^{\pm}$, these monomials form a {\it basis} of $(\cW^{\ell})^{U(1)}$, and the normal form is unique.

The filtration on $\cW^{\ell}$ restricts to a filtration $$(\cW^{\ell})^{U(1)}_{(0)} \subset (\cW^{\ell})^{U(1)}_{(1)} \subset \cdots $$ on $(\cW^{\ell})^{U(1)}$, where $(\cW^{\ell})^{U(1)}_{(r)} = (\cW^{\ell})^{U(1)} \cap \cW^{\ell}_{(r)}$. 
The $U(1)$-action descends to $\cV^{\ell} = \text{gr}(\cW^{\ell})$, and $$\text{gr}(\cW^{\ell})^{U(1)} \cong (\cV^{\ell})^{U(1)}.$$ Similarly, $U(1)$ acts on $\text{gr}(\cV^{\ell})$ and $\text{gr}((\cV^{\ell})^{U(1)}) \cong \text{gr}(\cV^{\ell})^{U(1)}$.
Clearly $\text{gr}(\cV^{\ell})^{U(1)}$ is generated as a ring by
$\{\partial^k J, \partial^k T, u_{i,j}|\ i, j,k\geq 0\}$, where $u_{i,j} = (\partial^i G^+) (\partial^j G^-)$. The ideal of relations among these generators is generated by \begin{equation} \label{ideal} u_{i,j} u_{k,l} - u_{i,l} u_{k,j}=0,\qquad 0 \leq i<k,\qquad 0 \leq j<l.\end{equation}
As a differential algebra under $\partial$, there is some redundancy in this generating set for $\text{gr}(\cV^{\ell})^{U(1)} $ since 
$$ \partial u_{i,j} = u_{i+1,j} +  u_{i,j+1}.$$ 
Letting $A_m$ be the span of $\{u_{i,j}|\ i+j = m\}$, note that for all $m>0$, $$A_m = \partial(A_{m-1}) \oplus \langle u_{0,m} \rangle,$$ where $\langle u_{0,m} \rangle$ denotes the span of $u_{0,m}$. Therefore 
$\{u_{i,j}|\ i,j \geq 0\}$ and $\{\partial^n u_{0,m}|\ m,n\geq 0\}$ span the same vector space, and $\{J, T, u_{0,m}|\ m\geq 0\}$ is a minimal generating set for $\text{gr}(\cV)^{U(1)}$ as a differential algebra. Define \begin{equation} \label{defuij} U_{i,j}=\ :\partial^i G^+ \partial^j G^-:\ \in (\cW^{\ell})^{U(1)}_{(2)},\end{equation} which has filtration degree $2$ and weight $i+j+3$.

\begin{lemma} $(\cW^{\ell})^{U(1)}$ is strongly generated as a vertex algebra by
\begin{equation} \label{gen} \{J, T, U_{0,m}|\ m\geq 0\}.\end{equation}
\end{lemma}
\begin{proof} Since $\{J, T, u_{0,m}|\ m\geq 0\}$ generates $\text{gr}(\cV^{\ell})^{U(1)} \cong \text{gr}((\cV^{\ell})^{U(1)})$ as a differential algebra, Lemma 3.6 of \cite{LL} shows that the corresponding set strongly generates $(\cV^{\ell})^{U(1)}$ as a vertex algebra. Since $(\cV^{\ell})^{U(1)} = \text{gr}(\cW^{\ell})^{U(1)} \cong \text{gr}((\cW^{\ell})^{U(1)})$, the result follows from Lemma \ref{reconlem}. \end{proof} 
In terms of the generating set \eqref{gen}, $(\cW^{\ell})^{U(1)}_{(2r)}$ is spanned by elements with at most $r$ of the fields $U_{0,m}$ and $(\cW^{\ell})^{U(1)}_{(2r+1)} = (\cW^{\ell})^{U(1)}_{(2r)}$.

Observe next that $(\cW^{\ell})^{U(1)}$ is not {\it freely} generated by \eqref{gen}. To see this, observe that 
\begin{equation} \label{rel:wt8} u_{0,0} u_{1,1} - u_{0,1} u_{1,0}=0 \end{equation} is the unique relation of the form \eqref{ideal} in $\text{gr}(\cV^{\ell})^{U(1)}$, of minimal weight $8$. The corresponding element $:U_{0,0} U_{1,1}: - :U_{0,1} U_{1,0}:$ of $(\cW^{\ell})^{U(1)}$ does not vanish due to \eqref{bp5}. However, it lies in the degree $2$ filtered piece $(\cW^{\ell})^{U(1)}_{(2)}$ and has the form 
\begin{equation} \label{fundrel} :U_{0,0} U_{1,1}: - :U_{0,1} U_{1,0}: \ = \frac{1}{60}\ell(2\ell - 1) U_{0,5} + \cdots.\end{equation} 
The remaining terms lie in $(\cW^{\ell})^{U(1)}_{(2)}$ and are normally ordered monomials in $\{J,T, U_{0,i}|\ 0\leq i \leq 4\}$ and their derivatives. For the reader's convenience, this relation is written down explicitly in the Appendix. 
Note that $$U_{1,0} = -U_{0,1} + \partial U_{0,0}, \qquad U_{1,1} =  \partial U_{0,1} - U_{0,2}.$$ 
Therefore the left side of \eqref{fundrel} is a normally ordered polynomial in $U_{0,0}, U_{0,1}, U_{0,2}$, so \eqref{fundrel} can be rewritten in the form
\begin{equation} \label{rel:wt8a} \frac{1}{60}\ell(2\ell - 1) U_{0,5} = P(J, T, U_{0,0}, U_{0,1}, U_{0,2},U_{0,3},U_{0,4}),\end{equation}
where $P$ is a normally ordered polynomial in $J, T, U_{0,0}, U_{0,1}, U_{0,2},U_{0,3},U_{0,4}$, and their derivatives.
We call this a {\it decoupling relation} since it allows $U_{0,5}$ to be expressed as a normally ordered polynomial in these fields whenever $\ell \neq 0$ or $1/2$. Since there are no relations among the generators $\{J,T, u_{i,j}|\ i,j \geq 0\}$ of $\text{gr}(\cV)^{U(1)}$ of weight less than $8$, there are no decoupling relations for $U_{0,0}, U_{0,1}, U_{0,2},U_{0,3},U_{0,4}$. We shall see that the coefficient of $U_{0,5}$ in \eqref{rel:wt8a} is {\it canonical} in the sense that it does not depend on any choices of normal ordering in $P$. The uniqueness of \eqref{rel:wt8} up to scalar multiples implies the uniqueness of \eqref{rel:wt8a}, so for $\ell = 0$ and $\ell = 1/2$, there is no decoupling relation for $U_{0,5}$.

In order to construct more decoupling relations of this kind, we introduce a certain invariant of elements of $(\cW^{\ell})_{(2)}^{U(1)}$. Given $\omega \in (\cW^{\ell})_{(2)}^{U(1)}$ of weight $n+7$, write $\omega$ in normal form. For $i = 0,1,\dots, n+4$, let $C_{n,i}(\omega)$ denote the coefficient of $:(\partial^iG^+)( \partial^{n+4-i} G^-):$ appearing in the normal form. Define \begin{equation} \label{defcn} C_n(\omega) = \sum_{i=0}^{n+4} (-1)^i C_{n,i}(\omega).\end{equation} Next, since $\{J,T,U_{0,i}|\ i\geq 0\}$ strongly generates $(\cW^{\ell})^{U(1)}$ and since $U_{0,i}$ has weight $i+3$, we may express $\omega$ as a normally ordered polynomial $P_{\omega}(J,T,U_{0,0}, U_{0,1},\dots, U_{0,n+4})$ in the subset $\{J,T,U_{0,i}|\ 0 \leq i \leq n+4\}$ and their derivatives. Since there exist relations among the generators $\{J,T,U_{0,i}|\ 0 \leq i \leq n+4\}$, as well as different choices of normal ordering, such an expression for $\omega$ is not unique. In particular, the coefficients of $\partial^i U_{0,n+4-i}$ for $i>0$ will depend on the choice of normal ordering in $P_{\omega}(J,T,U_{0,0}, U_{0,1},\dots, U_{0, n+4})$.

\begin{lemma} \label{indep} For any $\omega \in (\cW^{\ell})_{(2)}^{U(1)}$ of weight $n+7$, the coefficient of $U_{0,n+4}$ appearing in $P_{\omega}(J,T,U_{0,0}, U_{0,1},\dots, U_{0, n+4})$ is independent of all choices of normal ordering in $P_{\omega}$, and coincides with $(-1)^n C_n(\omega)$. \end{lemma} 

\begin{proof} Let $\cJ^{\ell} \subset (\cW^{\ell})^{U(1)}$ denote the subspace spanned by elements of the form $:a \partial b:$ with $a,b \in (\cW^{\ell})^{U(1)}$. It is well known \cite{YZh} that Zhu's commutative algebra $$C((\cW^{\ell})^{U(1)}) = (\cW^{\ell})^{U(1)} / \cJ^{\ell}$$ has generators corresponding to the strong generators $\{J,T,U_{0,n}|\ n\geq 0\}$ of $(\cW^{\ell})^{U(1)}$. In particular, given $\omega \in (\cW^{\ell})_{(2)}^{U(1)}$ of filtration degree $2$ and weight $n+7$, suppose we have two expressions $$\omega = P_{\omega}(J,T,U_{0,0}, U_{0,1},\dots, U_{0,n+4}) =  Q_{\omega}(J,T,U_{0,0}, U_{0,1},\dots, U_{0,n+4})$$ where $P,Q$ are normally ordered polynomials. Let $\tilde{P}_{\omega}$ and $\tilde{Q}_{\omega}$ denote the components of $P_{\omega}, Q_{\omega}$ which are linear combinations of $\partial^i U_{0,n+4-i}$ for $i=0,1,\dots n+4$. Then $\tilde{P}_{\omega} - \tilde{Q}_{\omega}$ lies in $\cJ^{\ell}$, and hence must be a total derivative. It follows that the coefficient of $U_{0,n+4}$ in $P_{\omega}$ and $Q_{\omega}$ is the same. The last statement follows from the fact that $$U_{i,n+4-i} =  \sum_{r=0}^i (-1)^{i +r} \binom{i}{r} \partial^r U_{0,n+4-r}.$$
\end{proof}

\begin{cor} The coefficient of $U_{0,5}$ in \eqref{fundrel} coincides with $-C_1(:U_{0,0} U_{1,1}: - :U_{0,1} U_{1,0}:)$ and is independent of all choices of normal ordering.
\end{cor}

\begin{thm} \label{u1invariant} For all $\ell \neq 0$ or $1/2$, $(\cW^{\ell})^{U(1)}$ has a minimal strong generating set $$\{J,T, U_{0,0}, U_{0,1},U_{0,2}, U_{0,3},U_{0,4}\},$$ and in particular is of type $\cW(1,2,3,4,5,6,7)$.\end{thm}

\begin{proof} It suffices to construct decoupling relations of the form $$U_{0,n+4} = P_n(J,T,U_{0,0}, U_{0,1},U_{0,2}, U_{0,3},U_{0,4}),$$ for all $n> 1$ and $\ell \neq 0$ or $1/2$, since we already have such a relation for $n=1$. Since $G^{\pm}$ commute in $\text{gr}(\cW^{\ell})$, $:U_{0,0} U_{1,n}: - :U_{0,n} U_{1,0}:$ lies in $\cW^{\ell}_{(2)}$, and we can express it as a normally ordered polynomial in the generators $J,T, U_{0,m}$ and their derivatives. We have
$$:U_{0,0} U_{1,n}: - :U_{0,n} U_{1,0}:\ = C_n U_{0,n+4} + \cdots,$$ where $C_n = C_n(:U_{0,0} U_{1,n}: - :U_{0,n} U_{1,0}:)$, and the remaining terms are normally ordered polynomials in $\{J,T, U_{0,0}, U_{0,1},\dots, U_{0,n+4-1}\}$ and their derivatives. It suffices to show that $C_n \neq 0$ for all $\ell \neq 0$ or $1/2$; the claim then follows by induction on $n$.

First, we have
$$:U_{0,0} U_{1,n}: - :U_{0,n} U_{1,0}: \ = \ :(:G^+ G^-:) (:(\partial G^+) \partial^n G^-:): - :G^+ G^- (\partial G^+) \partial^n G^-: $$ $$+ :G^+ G^- (\partial G^+) \partial^n G^-: - :G^+ G^- (\partial^n G^-) (\partial G^+): $$
$$+ :G^+ G^- (\partial^n G^-) (\partial G^+): - :G^+  (\partial^n G^-)G^- (\partial G^+): $$
$$+ :G^+  (\partial^n G^-)G^- (\partial G^+):  - :G^+  (\partial^n G^-)(\partial G^+)  G^-: $$
$$ +  :G^+  (\partial^n G^-)(\partial G^+) G^-:  - :(:G^+  (\partial^n G^-):) (:(\partial G^+)G^-:):.$$

It is immediate that $C_{n,i}(:G^+ G^- (\partial^n G^-) (\partial G^+): - :G^+  (\partial^n G^-)G^- (\partial G^+): ) = 0$. Let 
$$C^1_{n,i} = C_{n,i}\bigg(:(:G^+ G^-:) (:(\partial G^+) \partial^n G^-:): - :G^+ G^- (\partial G^+) \partial^n G^-:\bigg),$$
$$C^2_{n,i} = C_{n,i}\bigg(:G^+ G^- (\partial G^+) \partial^n G^-: - :G^+ G^- (\partial^n G^-) (\partial G^+): \bigg),$$
$$C^3_{n,i} = C_{n,i}\bigg(G^+  (\partial^n G^-)G^- (\partial G^+):  - :G^+  (\partial^n G^-)(\partial G^+)  G^-:  \bigg),$$
$$C^4_{n,i} = C_{n,i}\bigg( :G^+  (\partial^n G^-)(\partial G^+) G^-:  - :(:G^+  (\partial^n G^-):) (:(\partial G^+)G^-:): \bigg),$$
so that $C_n = \sum_{i=1}^4 \sum_{j=0}^{n+4} (-1)^j C^i_{n,j}$. Using the OPE formulas \eqref{bp1}-\eqref{bp5}, one can check that
$$C^1_{n,0} = 0, \qquad C^1_{n,1} = \frac{(3 + 2 \ell) (4 + 4 \ell + n + 2 \ell n)}{4 (2 + n) (3 + n)},\qquad C^1_{n,2} = - \frac{3 (3 + 2 \ell)}{4 (n + 2)},$$
$$C^1_{n,3} = \frac{3 + 2 \ell}{2 n + 2} - \frac{3 + 10 \ell + 6 n + 4 \ell n}{12 (1 + n)},$$ 
$$ C^1_{n,4} = -\frac{(3 + 2 \ell) (5 + 6 \ell)}{48} - \frac{3 + 10 \ell + 6 n + 4 \ell n}{48},$$
$$C^1_{n, j} = -\frac{(3 + 10 \ell + 6 n + 4 \ell n) n!}{2 (n + 4 - j)! j!},\qquad j = 5,\dots, n,$$

$$C^1_{n, n+1} = -\frac{3 + 10 \ell + 6 n + 4 \ell n}{12 (1 + n)},\qquad C^1_{n, n+2} = -\frac{3 + 10 \ell + 6 n + 4 \ell n}{4 (1 + n) (2 + n)},$$
$$C^1_{n, n+3} = -\frac{3 + 10 \ell + 6 n + 4 \ell n}{2 (1 + n) (2 + n) (3 + n)},\qquad C^1_{n, n+4} = -\frac{3 (5 + 2 \ell + 2 n)}{(1 + n) (2 + n) (3 + n) (4 + n)}.$$
Similarly, we have
$$C^2_{n,0} = \frac{18 - 4 \ell + 3 n + 2 \ell n}{2 (1 + n) (2 + n) (3 + n) (4 + n)},$$
$$C^2_{n,j} = - \frac{6 (n!)}{(n + 4 - j)! (j!)},\qquad j = 1,\dots, n$$
$$C^2_{n,n+1} = -\frac{1}{1 + n},\qquad C^2_{n,n+2} = -\frac{3}{(1 + n) (2 + n)},\qquad C^2_{n,n+3} = -\frac{6}{(1 + n) (2 + n) (3 + n)},$$ $$C^2_{n,n+4} = -\frac{-15 - 2 \ell - 6 n + 4 \ell n}{2 (1 + n) (2 + n) (3 + n) (4 + n)}.$$
Next, we have
$$C^3_{n,0} = -\frac{18 - 4 \ell + 3 n + 2 \ell n}{2 (1 + n) (2 + n) (3 + n) (4 + n)},\qquad C^3_{n,1} =  \frac{6}{(1 + n) (2 + n) (3 + n)},$$
$$C^3_{n,2} =   \frac{ 3}{(1 + n) (2 + n)}, \qquad C^3_{n,3} = \frac{1}{1 + n},\qquad C^3_{n,4} = -\frac{5}{16} - \frac{\ell}{24}, $$
$$C^3_{n,i} = 0, \qquad 5\leq i \leq n+4.$$
Finally, we have
$$C^4_{n,0} = 0,\qquad C^4_{n,1} = -\frac{3 (4 + n)}{4 (2 + n) (3 + n)} + \frac{\ell (-30 - 10 n + n^2)}{6 (2 + n) (3 + n)} -\frac{ \ell^2}{3},$$
$$C^4_{n,2} = \frac{3 (6 + 4 \ell - n + 2 \ell n)}{8 (2 + n)},\qquad C^4_{n,3} = -1,$$
$$C^4_{n,j} = 0 ,\qquad 4\leq j \leq n+2,$$
$$C^4_{n,n+3} = -\frac{(-1)^n (3 + 2 \ell)}{4 (3 + n)},$$
$$C^4_{n,n+4} = \frac{(-1)^n (45 + 40 \ell + 12 \ell^2 + 25 n + 47 \ell n + 22 \ell^2 n + 14 \ell n^2 + 12 \ell^2 n^2 - n^3 + \ell n^3 + 2 \ell^2 n^3)}{2 (1 + n) (2 + n) (3 + n) (4 + n)}.$$

Then $C_n(:U_{0,0} U_{1,n}: - :U_{0,n} U_{1,0}:) = \sum_{i=1}^4 \sum_{j=0}^{n+4} C^i_{n,j}$. Considerable cancellation occurs, and this turns out to be given by the simple expression 
$$C_n(:U_{0,0} U_{1,n}: - :U_{0,n} U_{1,0}:) =  \frac{n (n+7)}{4! (n+3)(n+4)}\ell (2\ell-1),$$ which is clearly nonzero for $\ell \neq 0$ or $1/2$. It follows that for all $n$, we have a relation $$:U_{0,0} U_{1,n}: - :U_{0,n} U_{1,0}:  = (-1)^{n} \frac{n (n+7)}{4! (n+3)(n+4)}\ell (2\ell-1) U_{0,n+4}+ \cdots,$$ where the remaining terms depend on $\{J,T, U_{0,0}, U_{0,1},U_{0,2},U_{0,3},U_{0,4}\}$ and their derivatives.
\end{proof}

\section{The structure of $\text{Com}(\cH, \cW^{\ell})$} \label{sect:structuregeneric}
Let $\cH \subset \cW^{\ell}$ denote the copy of the Heisenberg vertex algebra generated by $J$, and let $\cC^{\ell}$ denote the commutant $\text{Com}(\cH, \cW^{\ell})$. Clearly we have $$(\cW^{\ell})^{U(1)}  \cong \cH \otimes \cC^{\ell}$$ and $\cC^{\ell}$ has a Virasoro element $$T^{\cC} = T - T^{\cH},\qquad T^{\cH} = \frac{3}{4\ell} :JJ:.$$

\begin{thm} \label{thm:generic-coset} For $0\leq i \leq 4$, and $\ell \neq 0$, there exist correction terms $\omega_i \in (\cW^{\ell})^{U(1)}_{(2)}$ such that $U^{\cC}_i = U_{0,i} + \omega_i$ lies in $\cC^{\ell}$. Therefore $\cC^{\ell}$ has a minimal strong generating set $\{T^{\cC}, U^{\cC}_i|\ 0\leq i \leq 4\}$, and is therefore of type $\cW(2,3,4,5,6,7)$ for $\ell \neq 0$ or $1/2$. Moreover, when $\ell$ is also not a root of any the following quadratic polynomials
$$60 x^2 - 104 x -51,\qquad  28 x^2 - 104 x -107 ,\qquad  4x^2 -20 x -23, \qquad   24 x^2 - 22 x +9 ,$$ $$ 84 x^2 - 220 x -183,\qquad 6 x^2 - 29 x-33,\qquad 60 x^2- 52 x + 27,\qquad 132 x^2 - 832 x -1017,$$
there exist a unique correction term $\omega_i \in (\cW^{\ell})^{U(1)}_{(2)}$ such that $U^{\cC}_i = U_{0,i} + \omega_i$ lies in $\cC^{\ell}$ and is primary of weight $i+3$ with respect to $T^{\cC}$.  \end{thm}

\begin{proof} The existence and uniqueness of the terms $\omega_i$ satisfying the above properties is a straightforward computer calculation. In the explicit formulas for $U_{0,i}^{\cC}$ for $i=1,2,3,4$, all coefficients of normally ordered monomials in $J, T, U_{0,i}$ are rational functions of $\ell$, whose denominator only contains the above quadratic factors together with $\ell$. \end{proof}

\subsection{An alternative realization of $\cC^{\ell}$}\label{sect:altreal}

Feigin and Semikhatov \cite{FS} constructed a family of vertex algebras which they call $\cW^{(2)}_n$-algebras, and the case $n=3$ is the Bershadsky-Polyakov algebra. Their first construction is as a centralizer of some quantum super group action and the second is in terms of a coset associated to $V_{k}\left(\gs\gl(n|1)\right)$, that is the affine vertex super algebra of $\gs\gl(n|1)$ at level $k$. One consequence is
\begin{thm} \label{FScosetreal}
Let $k$ and $k'$ be complex numbers related by $(k+2)(k'+2)=1$. Then for generic values of $k$,
$$
\text{Com}(V_{k'}(\gg\gl_3),  V_{k'}(\gs\gl(3|1)))\cong \cC^{\ell},\qquad \ell = k+3/2.$$ \end{thm}
 \begin{proof}
Feigin and Semikhatov consider $V_{k'}(\gs\gl(3|1))\otimes V_L$, where $V_L$ is a certain rank one lattice VOA.
They find that the coset by $V_{k'}(\gs\gl_3)\otimes \mathcal H$ for a certain rank one Heisenberg sub VOA $\mathcal H$ contains as subalgebra a VOA with same OPE algebra as $\cW^{\ell}$ for $\ell = k +3/2$. Since the Bershadsky-Polyakov algebra is generically simple \cite{ArI,ArII} the two must be isomorphic for generic level $k$.
Furthermore,
$$
\text{Com}( V_{k'}(\gg\gl_3) , V_{k'}(\gs\gl(3|1)))\cong \text{Com}(V_{k'}(\gg\gl_3) \otimes \mathcal H, V_{k'}(\gs\gl(3|1))\otimes V_L),
$$
since the Heisenberg coset of a rank one lattice VOA is one-dimensional. 
In \cite{CLI} we proved that the left-hand side for generic level is of type $\cW(2, 3, 4, 5, 6, 7)$. This completes the proof, since if $\cW\subset \cV$ is a vertex subalgebra with the same strong generating set as $\cV$, we must have $\cW=\cV$. 
\end{proof}

\section{The structure of $\text{Com}(\cH, \cW_{\ell})$} \label{sect:structuresimple}

For $\ell \neq -3/2$, let $\cI_{\ell}$ denote the maximal proper ideal of $\cW^{\ell}$, and let $\cW_{\ell}$ denote the simple quotient $\cW^{\ell} / \cI_{\ell}$. Let $\cC_{\ell}$ denote the quotient $\cC^{\ell} / (\cI_{\ell} \cap \cC^{\ell})$.

\begin{lemma} $\cC_{\ell} = \text{Com}(\cH, \cW_{\ell})$ where $\cH \subset \cW_{\ell}$ is the Heisenberg algebra generated by $J$. In particular, $\cC_{\ell}$ is simple.
\end{lemma}

\begin{proof} This is immediate from the fact that $\cW^{\ell}$ is completely reducible as an $\cH$-module, and $\cI_{\ell}$ is an $\cH$-submodule of $\cW^{\ell}$.
\end{proof}

As shown by Arakawa \cite{ArII}, when $\ell$ is a positive integer, $\cI_{\ell}$ is generated by $(G^{\pm})^{2\ell+1}$, and $\cW_{\ell}$ is $C_2$-cofinite and rational.

\begin{thm} \label{thm:latticesubvoa} For all positive integers $\ell$, $\cW_{\ell}$ contains a rank one lattice vertex algebra $V_L$ with generators $J, (G^{\pm})^{2\ell}$, where $L = \sqrt{6\ell} \mathbb{Z}$. Moreover, $\text{Com}(\text{Com}(\cH, \cW_{\ell})) = V_L$ so $V_L$ and $\cC_{\ell}$ form a Howe pair inside $\cW_{\ell}$. 
\end{thm}

\begin{proof} In $\cW^{\ell}$ we have $$J(z) (G^{\pm})^{2\ell}(w) \sim \pm 2\ell (G^{\pm})^{2\ell}(w)(z-w)^{-1},$$ so $(G^{\pm})^{2\ell}$ are both primary with respect to $J$ in $\cW^{\ell}$, and similarly in $\cW_{\ell}$. It suffices to show that $(G^{\pm})^{2\ell}$ both lie in the double commutant $\text{Com}(\text{Com}(\cH, \cW_{\ell}))$. Equivalently, we need to show that $(G^{\pm})^{2\ell}$ both commute with $T^{\cC} = T - T^{\cH}$ in $\cW_{\ell}$. We have the following OPE relations in $\cW^{\ell}$.
$$T^{\cH}_0 (G^{\pm})^{2\ell} = 3\ell (G^{\pm})^{2\ell},\qquad T_0 (G^{\pm})^{2\ell} = 3\ell (G^{\pm})^{2\ell},$$
so $T^{\cC}_0 (G^{\pm})^{2\ell} = 0$ in $\cW^{\ell}$. Similarly, $$T^{\cH}_{-1} (G^{\pm})^{2\ell} = \pm 3 :J (G^{\pm})^{2\ell}:,\qquad T_{-1} (G^{\pm})^{2\ell} = \partial (G^{\pm})^{2\ell},$$ so
$T^{\cC}_{-1} (G^{\pm})^{2\ell} = \partial (G^{\pm})^{2\ell} \mp 3 :J (G^{\pm})^{2\ell}:$. Finally, $$G^- \circ_1 (G^+)^{2\ell+1} = \frac{(2\ell+1)^2}{2 }(\partial ((G^+)^{2\ell}) - 3 :J (G^+)^{2\ell}:),$$ so $\partial ((G^+)^{2\ell}) - 3 :J (G^+)^{2\ell}: \ = 0$ in $\cW_{\ell}$. Similarly, $$G^+ \circ_1 (G^-)^{2\ell+1} = -\frac{(2\ell+1)^2}{2 }(\partial ((G^-)^{2\ell}) + 3 :J (G^-)^{2\ell}:),$$ so
$\partial ((G^-)^{2\ell}) + 3 :J (G^-)^{2\ell}: \ = 0$ in $\cW_{\ell}$. It is straightforward to check that for all $k\geq 1$, $T_k(G^{\pm})^{2 \ell} = 0$ and $T^{\cH}_k(G^{\pm})^{2 \ell} = 0$ in $\cW^{\ell}$, and hence in $\cW_{\ell}$ as well. It follows that $(G^{\pm})^{2\ell}$ commutes with $T^{\cC}$ in $\cW_{\ell}$. \end{proof}

\begin{cor} For all positive integers $\ell$, $\cC_\ell$ is $C_2$-cofinite.
\end{cor}
\begin{proof} This follows from Corollary 2 of \cite{MiI}. \end{proof}

Since $\cC^{\ell}$ has strong generators $\{T^{\cC}, U_{0,0}^{\cC}, U_{0,1}^{\cC},U_{0,2}^{\cC},U_{0,3}^{\cC},U_{0,4}^{\cC}\}$ for $\ell \neq 0,\frac{1}{2}$, it follows that for $\ell = 1,2,\dots$, $\cC_{\ell}$ is generated by the images of these fields in $\cW_{\ell}$, which we also denote by $\{T^{\cC}, U_{0,0}^{\cC}, U_{0,1}^{\cC},U_{0,2}^{\cC},U_{0,3}^{\cC},U_{0,4}^{\cC}\}$ by abuse of notation. For $\ell \geq 4$ this is a minimal strong generating set, but for $\ell = 1,2,3$ there are additional decoupling relations and a smaller generating set suffices.

Recently, Kawasetsu proved the following result in the case $\ell = 1$ \cite{Ka}.

\begin{thm} (Kawasetsu) $\cC_1$ is isomorphic to the rational Virasoro vertex algebra generated by $T^{\cC}$ of central charge $c = -3/5$.\end{thm}

An alternative proof of this result can be given as follows. One can check by computer calculation that for $i=1,2,3,4$, $U^{\cC}_{0,i} \in \cC^1$ lies in $\cI_1 \cap \cC^1$, where $\cI_1 \subset \cW^1$ denotes the maximal ideal. In particular, $U^{\cC}_{0,i} =0$ as an element of $\cC_1$, so $\cC_1$ is strongly generated by $T^{\cC}$. Being simple it must be the rational Virasoro algebra with $c=-3/5$.

Similarly, in the case $\ell = 2$, one can check that $U^{\cC}_{0,2}$ vanishes in $\cC_2$, and that $U^{\cC}_{0,3}$ and $U^{\cC}_{0,4}$ can be expressed as normally ordered polynomials in $\{T^{\cC}, U^{\cC}_{0,0}, U^{\cC}_{0,1}\}$. Since there are no relations of weight less than $5$, it follows that $\cC_{2}$ is of type $\cW(2,3,4)$ and has a minimal strong generating set $\{T^{\cC}, U_{0,0}^{\cC}, U_{0,1}^{\cC}\}$. In the case $\ell = 3$, one checks that $U_{0,4}^{\cC}$ can be expressed as a normally ordered polynomial in $\{T^{\cC}, U_{0,i}^{\cC}|\ i=0,1,2,3\}$. Since there are no relations of weight less than $6$, $\cC_3$  is of type $\cW(2,3,4,5,6)$. Finally, for $\ell >3$, the element of $\cI_{\ell} \cap \cC^{\ell}$ of minimal weight is $(G^+ \circ_1)^{2\ell + 1}  ((G^-)^{2\ell +1})$, which has weight $2\ell+1$. Therefore none of the generators of $\cC_{\ell}$ decouple, so $\cC_{\ell}$ is of type $\cW(2,3,4,5,6,7)$.

In this paper our focus is on $\cC_{\ell}$ when $\ell$ is a positive integer, but there are other interesting levels as well. For example,
\begin{lemma}
Let $\cA(3)$ be the rank three symplectic fermion algebra. Then
$$\cA(3)^{GL(3)}\cong \mathcal \cC_{\ell},\qquad \ell = -1/2.$$
\end{lemma}

\begin{proof}
$\cA(3)^{GL(3)}$ is of type $\cW(2, 3, 4, 5, 6,7)$ \cite{CLII} and as an orbifold of a simple vertex algebra it is itself simple \cite{DM}. Using the notion of deformable family of vertex algebras introduced in \cite{CLIII}, it has been proven \cite{CLI} that
$$\lim_{k'\rightarrow \infty }\text{Com}(V_{k'}(\gg\gl_3), V_{k'}(\gs\gl(3|1)))\cong \cA(3)^{GL(3)}.$$
By Theorem \ref{FScosetreal}, both $\cA(3)^{GL(3)}$ and $\cC^{-1/2}$ have the same OPE algebra, and hence their simple quotients must coincide. 
\end{proof}

\section{$\cW$-algebras of type $A$} \label{sect:WtypeA}
In this section, we prove our main result that for all $\ell \geq 1$, $\cC_{\ell}$ is isomorphic to the principal, rational $\cW(\gs\gl_{2l})$-algebra with central charge $c=-3(2\ell-1)^2/(2\ell +3)$, which has level $(2\ell+3)/(2\ell+1)-2\ell$. In the case $\ell = 2$, using the fact that $\cC_2$ is of type $\cW(2,3,4)$, one can verify by computer calculation that the OPE relations of the generators coincide with the OPE relations of $\cW(\gs\gl_4)$, which appear explicitly in \cite{CZh}. However, for $\ell >2$ it is not practical to check this directly, and a more conceptual approach is needed.

Let $\cV^\ell$ be the universal $\cW(\gs\gl_{2\ell})$ at level $(2\ell+3)/(2\ell+1)-2\ell$ and let $\cV_\ell$ be its simple quotient. 
 Its rationality has been proven \cite{ArV}.
 Simple modules of $\cV_\ell$ are labelled by integrable positive  weights of $\widehat{\gs\gl}_{2\ell}$ at level $3$. 
 We denote the fundamental weights of $\widehat{\gs\gl}_{2\ell}$ by $\Lambda_0, \Lambda_1, \dots, \Lambda_{2\ell-1}$.
By Theorem 4.3 of \cite{FKW} (see
 also \cite{AvE}) $\cV_\ell$ has a simple current $\mathbb{L}_{3\Lam_1}$ of order $2\ell$ and both $\mathbb{L}_{3\Lam_1}$ and its inverse $\mathbb{L}_{3\Lam_{2\ell-1}}$ are generated as $\cV_\ell$ modules by a single lowest-weight state of conformal dimension $\frac{3}{2}-\frac{3}{4\ell}$. Further the conformal dimension of the lowest-weight-state of both $\mathbb{L}_{3\Lam_s}$ and $\mathbb{L}_{3\Lam_{2\ell-s}}$ for $1\leq s\leq\ell$ is $\frac{3s}{2}-\frac{3s^2}{4\ell}$.
Let $L=\sqrt{6\ell}\Z$ and $V_L$ the lattice vertex algebra associated to $L$. Then $V_L$ has a simple current $V_{L+\frac{3}{\sqrt{6\ell}}}$ of order $2\ell$ with lowest-weight state of conformal dimension $\frac{3}{4\ell}$. It follows that 
\begin{equation}\label{eq:decomposition}
\cB_\ell:= \bigoplus_{s=0}^{2\ell-1}\mathbb{L}_{3\Lam_s} \otimes V_{L+\frac{3s}{\sqrt{6\ell}}}
\end{equation}
is a simple current extension of $\cV_\ell\otimes V_L$. A priori it is not clear that $\cB_\ell$ is a VOA, but one purpose of \cite{CKL} was to develop a framework (extending earlier works \cite{C, LLY}) that decides this type of question.
\begin{lemma}
$\cB_\ell$ can be given the structure of a VOA
such
that the natural map
$\cV_\ell\otimes V_L\hookrightarrow \cB_\ell$ is a conformal embedding.
\end{lemma}
\begin{proof}
This follows from \cite{CKL}:
a simple current extension is either a VOA or a super VOA by Theorem 3.12 of \cite{CKL}, and if the tensor category of the underlying VOA is modular, the parity question is decided by the categorical dimension of the generating simple current by Corollary 2.8 of \cite{CKL}. This follows directly from Huang's work on the Verlinde formula \cite{HuII}, and is explained in \cite{CKL}. Since a lattice VOA is a unitary VOA, any of its simple currents has categorical dimension one, so we have to determine the categorical dimension of $\mathbb{L}_{3\Lam_1}$. By Proposition 4.2  together with Proposition 1.1  of \cite{FKW} the categorical dimension is given by
\[
\text{qdim}(\mathbb{L}_{3\Lam_1}) = \frac{S_{3\Lambda_1, 3\Lambda_0}}{S_{3\Lambda_0, 3\Lambda_0}}=\epsilon(3\Lambda_1) = (-1)^{2(\Lambda_1|\rho)}.
\]
Since $n=2\ell$ and $\alpha_1, \dots, \alpha_{n-1}$ are the simple roots of $\gs\gl_{n}$,
\[
\Lambda_1 = \frac{1}{n} \left( (n-1) \alpha_1 +(n-2) \alpha_2 + \dots + \alpha_{n-1}\right),
\] 
so that 
\[
(\Lambda_1 | \rho) = \frac{1}{n} \frac{n(n-1)}{2} = \frac{(n-1)}{2}.
\]
It follows that qdim$(\mathbb{L}_{3\Lam_1})=(-1)^{(n-1)}=-1$ since $n=2\ell$ is even. 
Odd quantum dimension and odd twist (half-integer conformal dimension) of a simple current imply that the extension is a VOA by Corollary 2.8 together with Theorems 3.9 and 3.12 of \cite{CKL}.
\end{proof}
We would like to show that $\cB_\ell\cong\cW_\ell$. Denote the vertex operator associated to the vector of lowest conformal dimension of $\mathbb{L}_{3\Lam_1} \otimes V_{L+\frac{3}{\sqrt{6\ell}}}$ by $F^+$ and the one associated to the vector of lowest conformal dimension of $\mathbb{L}_{3\Lam_{2\ell-1}} \otimes V_{L+\frac{6\ell-3}{\sqrt{6\ell}}}$ by $F^-$.

\begin{lemma}
The subalgebra of $\cB_\ell$ generated under operator product algebra by $F^\pm$ has the same OPE algebra as $\cW^\ell$. 
\end{lemma}
\begin{proof}
Let 
\[
J(z) =\sum_{n\in \Z } J_n z^{-n-1}, \qquad G^\pm(z) =\sum_{n\in \Z } G^\pm_n z^{-n-3/2}, \qquad
T^{\cC}(z) =\sum_{n\in \Z } L^c_n z^{-n-2}  
\]
be the mode expansion of the corresponding fields of $\cW^\ell$. Here $T^{\cC}$ denotes the Virasoro field of $\cC^\ell$. It is related to the Virasoro field of $\cW^\ell$ by $T=T^{\cC}+ \frac{4}{3\ell} (J)^2$ where $(J)^2$ denotes the normal ordered product of $J$ with itself. 
We can thus read off the commutation relations of the mode algebra from the operator product algebra of the fields given in Section 3, see also \cite{ArV}. 
These are 
\begin{equation}\nonumber
\begin{split}
[J_m, J_n] &= \frac{2\ell}{3} \delta_{n+m, 0}\\
[J_m, G^\pm_n] &= \pm G^\pm_{n+m} \\
[G^+, G^-]&= \frac{9}{4\ell} \left(\ell-\frac{1}{2}\right) (J^2)_{n+m} +\frac{3}{2}  \left(\ell-\frac{1}{2}\right)\left(m-n-1\right) J_{m+n}
-\left(\ell+\frac{3}{2}\right) T^{\cC}_{n+m} +\\
&\quad \ell \left(\ell-\frac{1}{2}\right) m(m+1) \delta_{n+m, 0}.
\end{split}
\end{equation}
Here we omit the Virasoro algebra relations as they are standard. 
We want to show that the modes of the fields $F^\pm$ together with those of the Heisenberg field $H$ of $V_L$ and the Virasoro field $S$ of $V_\ell$ have the same commutation relations. 
Let 
\[
H(z) =\sum_{n\in \Z } H_n z^{-n-1}, \qquad F^\pm(z) =\sum_{n\in \Z } F^\pm_n z^{-n-3/2}, \qquad
S(z) =\sum_{n\in \Z } S_n z^{-n-2}  
\]
be the mode expansion of the corresponding fields.
We can normalize $H$, such that 
\[
H(z) F^\pm(w) \sim  \pm (z-w)^{-1} F^\pm(w), 
\]
with this normalization the operator product of $H$ with itself is
\[
H(z) H(w) \sim \frac{2\ell}{3} (z-w)^{-2}. 
\]
The following commutation relations follow
\[
[H_m, H_n] = \frac{2\ell}{3} \delta_{n+m, 0},\qquad
[H_m, F^\pm_n] = \pm F^\pm_{n+m}.
\]
Since the state of lowest conformal dimension of both $\mathbb{L}_{3\Lam_2} \otimes V_{L+\frac{6}{\sqrt{6\ell}}}$ and
$\mathbb{L}_{3\Lam_{2\ell -2}} \otimes V_{L+\frac{6\ell-6}{\sqrt{6\ell}}}$ has dimension three the operator product of $F^+$ with itself as well as the one of $F^-$ with itself is regular. The operator products involving the Virasoro fields are of course the same as the corresponding ones of $\cW^\ell$. It remains to compute the operator product of $F^+$ with $F^-$. The known operator product algebra of lattice VOAs leaves us with two free coefficients $a, b$ for the moment:
\begin{equation}\nonumber
\begin{split}
F^+(z) F^-(w) &\sim (z-w)^{-3} a \left( 1 +\frac{3}{2\ell} (z-w)H + (z-w)^2 \left( \frac{1}{2} \left(\frac{3}{2\ell}\right)^2 (H^2)+\frac{1}{2} \frac{3}{2\ell}\partial H\right)\right) + \\
&\quad b(z-w)^{-1} S.
\end{split}
\end{equation}
We will see that these coefficients are uniquely determined by the Jacobi identity for the Lie algebra of the modes. 
First, $a\neq 0$ by Proposition 4.1 of \cite{CKL}. We can now rescale the fields $F^\pm$ so that $a=2\ell (\ell-\frac{1}{2})$. 
It remains to show that $b=-(\ell+\frac{3}{2})$. Set $b=-(\ell+\frac{3}{2})+\lambda$. Then using that the Jacobi identity of $\cW^\ell$ holds we get
\begin{equation}\nonumber
\begin{split}
0 &= [R_n, [F^+_m, F^-_r]] +[F^+_m, [F^-_r, R_n]] + [F^-_r, [R_n, F^+_m]]
= \lambda \frac{c_\ell}{12} (n^3-n) \delta_{n+m+r, 0} .  
\end{split}
\end{equation}
Hence $\lambda=0$. We thus have shown that the mode algebra of $\cW^\ell$ and the VOA generated by $F^\pm$ coincide, hence their operator product algebras coincide.
\end{proof}
\begin{thm}\label{thm:typeA}
The coset $\cC_\ell$ is isomorphic to the simple rational $\cW$-algebra of $\gs\gl_{2\ell}$ , $\cW(\gs\gl_{2\ell})$ at level $(2\ell+3)/(2\ell+1)-2\ell$.
\end{thm}
\begin{proof}
By the previous lemma, $\cB_\ell$ is a module for $\cW^\ell$ and it contains a homomorphic image of $\cW^\ell$ and hence the generators of weights $2, 3, 4, 5, 6, 7$ that strongly generate the corresponding homomorphic image of $\cC^\ell$. But under operator products, already the dimension two and three fields of $\cV_\ell$ generate $\cV_\ell$, hence  the homomorphic image of $\cC^\ell$ must be $\cV_\ell$. Since it is simple, we have $\cC_\ell \cong \cV_\ell$.
\end{proof}
\begin{cor}
$\cW_\ell$ is rational and $C_2$-cofinite.
\end{cor}
\begin{proof}
$\cW_\ell$ is a simple current extension of a rational and $C_2$-cofinite VOA of CFT-type, hence itself has these properties \cite{Y}.
\end{proof}
In other words, we have found a very different proof of the main result of \cite{ArII}.
\begin{cor}
The simple rational $\cW$-algebra of $\gs\gl_{2\ell}$ , $\cW(\gs\gl_{2\ell})$ at level $(2\ell+3)/(2\ell+1)-2\ell$, is a $\cW$-algebra
of type $\cW(2, 3, 4, 5, 6, 7)$ for $\ell>3$. For $\ell=1$ it is of type $\cW(2)$, for $\ell=2$ it is of type $\cW(2, 3, 4)$ and for $\ell=3$ it is of type $\cW(2, 3, 4, 5, 6)$.
\end{cor}
\begin{proof}
We have just proven that the simple rational $\cW$-algebra of $\gs\gl_{2\ell}$, $\cW(\gs\gl_{2\ell})$ at level $(2\ell+3)/(2\ell+1)-2\ell$,
is isomorphic to $\cC_\ell$, but the latter is of the corresponding type as proven in the last section. 
\end{proof}

\section{A character identity} \label{sect:character}

We first recall the known characters of $\cW_\ell$ and $\cV_\ell$ and their simple currents $\mathbb{L}_{3\Lam_s}$.
In this section $q=e(\tau), z=e(u)$ and $\tau$ in the upper half of the complex plane, and $u\in\C$.

\subsection{Characters of rational Bershadsky-Polyakov algebras}

Recall that $\ell=k+3/2$ and $\cW_\ell=\cW_{k}(\mf{sl}_3,f_{\theta})$, the simple Bershadsky-Polyakov algebra at level
$k$.
We have
\begin{align*}
 \cW_\ell=H^{0}_{DS,f_{\theta}}(L_{k}(\mf{sl}_3))
\end{align*}
provided that $k\not\in \Z_{\geq 0}$ (\cite{ArI}).
Here $H^{0}_{DS,f_{\theta}}(?)$ is the Drinfeld-Sokolov reduction with
respect to the nilpotent element $f_{\theta}$ (\cite{KRW})
and $L_k(\mf{sl}_3)$ is the simple affine vertex algebra associate with
$\mf{sl}_3$ at level $k$.
This gives
\begin{align*}
 \on{ch} \cW_\ell:&=\on{tr}_{\cW_\ell}\left(z^{J_0}q^{L_0-\frac{c}{24}}\right)\\
&= q^{-\frac{c}{24}}\left(\left.\on{ch} L_k(\mf{sl}_3)\prod_{\alpha\in
 \hat{\Delta}_+^{re}
}(1-e^{-\alpha})\frac{1}{\prod\limits_{n\geq
 0}(1-e^{-\alpha_1-n\delta})(1-e^{-\alpha_2-n\delta})
}\right)\right\vert_{\substack{e^{-\delta}\ra q,\\
 e^{\alpha_1}\ra z^{-1}q^{-1/2},\\
e^{\alpha_2}\ra zq^{-1/2}
 }}
\end{align*}

On the other hand if  $k$ is admissible we have \cite{KWI}
\begin{align*}
\on{ch} L(k \Lam_0)=\frac{\sum\limits_{w\in \widehat{W}(k\Lam_0)}(-1)^{\ell_{k\Lam_0}(w)}e^{w\circ
k\Lam_0}}{\prod\limits_{\alpha\in \Delta_+}(1-e^{-\alpha})^{\dim \widehat{\fing}_{\alpha}}},
\end{align*}
where $\widehat{W}(k\Lam_0)$ is the integrable Weyl group of $k\Lam_0$
generated by reflections $s_{\alpha}$ such that $\bra
k\Lam_0,\alpha^{\vee}\ket \in \Z$
and $\ell_{k\Lam_0}$ is the length function of $\widehat{W}(k\Lam_0)$.
The nontrivial rational levels are $k=p/2-3=\ell-3/2$ with $p=5,7,9,\dots$ and thus $\ell=1, 2, 3, \dots$
Then 
$k\Lam_0$ is admissible 
and 
\begin{align*}
 \widehat{W}(k\Lam_0)=W\ltimes 2 Q^{\vee}(\cong \widehat{W}=W\ltimes Q^{\vee}),
\end{align*}
where $W=\mf{S}_3$ and $Q^{\vee}$ is the coroot lattice of $\fing=\mf{sl}_3$.
Thus
\begin{align*}
 \on{ch} {\cW_\ell}= \frac{1}{\eta(\tau)\vartheta(\tau, u)}
 \left(\left.\sum\limits_{w\in W\ltimes 2 Q^{\vee}}e^{w\circ
 k\Lam_0}
\right)\right\vert_{\substack{e^{-\delta}\ra q,\\
 e^{\alpha_1}\ra z^{-1}q^{-1/2},\\
e^{\alpha_2}\ra zq^{-1/2}}}
\end{align*}
where
\[
\vartheta(\tau, u) =\prod\limits_{i\geq
 1}(1-q^i)\prod_{j\geq 1}(1-zq^{1/2+j})(1-z^{-1}q^{1/2+j}),\qquad
 \eta(\tau) =q^{\frac{1}{24}}\prod\limits_{i\geq
 1}(1-q^i). 
\]

\subsection{Characters of 
simple modules over $\cW_k(\mf{sl}_n)$}
Let $\cW_k(\mf{sl}_n)$ be the simple principal $\cW$-algebra
associated with $\mf{sl}_n$ at level $k$.
We consider the special case $k+n=\frac{n+p}{n+1}$ (for us $p=3$).
Then \cite{FKW,ArIII,ArV},
\begin{align*}
\widehat{P}_+^{p}\isomap  \{\text{simple $\cW^k(\mf{sl}_n)$-modules}\},
\quad \lam\mapsto 
\mathbb{L}_{\lam}:=
H^0_{DS,f_{prin}}(L(\lam+(k-p)\Lam_0)),
\end{align*}
where $\widehat{P}_+^p$ is the set of dominant integrable  weight of $\fing=\mf{sl}_n$
level $p$,
and
\begin{align*}
 \on{ch} \mathbb{L}_{\lam}=\frac{1}{\eta(\tau)^{n-1}}\sum_{w\in \widehat{W}}\epsilon(w)
q^{\frac{(n+p)(n+1)}{2}\left\vert\frac{w(\lam+\rho)}{n+p}-\frac{\rho}{n+1}\right\vert^2}.
\end{align*}
Here $\widehat{W}$ is the affine Weyl group of $\widehat{\mf{sl}}_n$.

\subsection{Lattice VOA characters}

We also need the characters of the lattice VOA $V_L$, they are just quotients of theta and eta functions:
\[
\ch[V_{L+{\frac{s}{\sqrt{6\ell}}}}](\tau, u) = \frac{1}{\eta(\tau)}\sum_{n\in \Z} q^{\frac{1}{2}\left( \sqrt{6\ell}n +\frac{s}{\sqrt{6\ell}}\right)^2}z^{3\ell\left(n +\frac{s}{3\ell}\right)} =:\frac{\theta_{s}(\tau, u)}{\eta(\tau)}.
\]

\subsection{Decomposing the character of $\cW_\ell$}
We can now put everything together to decompose the character of the VOA $\cW_\ell$.
\begin{thm}
The character of the simple current $U^s$ of the type $A$ $\cW$-algebra $\cV_\ell$ is in terms of the character of the simple Bershadsky-Polyakov algebra $\cW_\ell$
\[
\ch[\mathbb{L}_{3\Lam_s}](\tau) = \frac{1}{2\ell} \frac{\eta(\tau)}{\theta_{s}(\tau, 0)} \sum_{t=0}^{2\ell-1} e^{-2\pi i ts} \ch[\cW_\ell]\left(\tau, \frac{t}{3\ell}\right). 
\]
\end{thm}
\begin{proof}
By Theorem \ref{thm:typeA} together with equation \eqref{eq:decomposition} we have that
\[
\ch[\cW_\ell]\left(\tau, u\right)= \sum_{s=0}^{2\ell-1} \ch[\mathbb{L}_{3\Lam_s}](\tau) \ch[V_{L+{\frac{s}{\sqrt{6\ell}}}}](\tau, u).
\]
The theorem now follows directly by using 
$e^{-2\pi i ts}\theta_{s'}(\tau, u+\frac{t}{3\ell}) = e^{-2\pi i t(s-s')}\theta_{s'}(\tau, u)$
which implies
\[
\frac{1}{2\ell}\sum_{t=0}^{2\ell-1} e^{-2\pi i ts}\theta_{s'}(\tau, u+\frac{t}{3\ell})
= \frac{1}{2\ell}\sum_{t=0}^{2\ell-1} e^{-2\pi i t(s-s')}\theta_{s;}(\tau, u) 
= \delta_{s, s'}\theta_{s}(\tau, u).
\]
\end{proof}
Plugging in the explicit form of the characters, we get the identities of different lattice theta functions:
\begin{cor}
For $s=0, \dots, 2\ell-1$ the following identity holds
\[
\sum_{w\in \widehat{W}}\epsilon(w)
q^{\frac{(n+p)(n+1)}{2}\left\vert\frac{w(3\Lam_s+\rho)}{n+p}-\frac{\rho}{n+1}\right\vert^2}=\frac{1}{2\ell} \frac{\eta(\tau)^{n-1}}{\theta_{s}(\tau, 0)} \sum_{t=0}^{2\ell-1} \frac{e^{-2\pi i ts}}{\vartheta\left(\tau, \frac{t}{3\ell}\right)}\left(\left.\sum\limits_{w\in W\ltimes 2 Q^{\vee}}e^{w\circ
 k\Lam_0}
\right)\right\vert_{\substack{e^{-\delta}\ra q,\\
 e^{\alpha_1}\ra e^{-\frac{2\pi i t}{3\ell}}q^{-1/2},\\
e^{\alpha_2}\ra e^{\frac{2\pi i t}{3\ell}}q^{-1/2}}}
\]
where on the left hand side we have the affine Weyl group of $\gs\gl_n$ while on the right hand side it is the Weyl group and coroot lattice of $\gs\gl_3$. 
\end{cor}

\section{A family of super $\cW$-algebras}\label{sect:newvsa}

We consider the tensor product $\cW^{\ell} \otimes \cE$, where $\cE$ is the rank one $bc$-system, or equivalently the lattice vertex algebra $V_{\mathbb{Z}}$. It has odd generators $b,c$ satisfying $$b(z) c(w) \sim (z-w)^{-1},$$ a Heisenberg element $J^{\cE} = -:bc:$ and a Virasoro element $T^{\cE} = -:b\partial c: + :(\partial b) c:$ of central charge $1$, under which $b,c$ are primary of weight $1/2$.

Let $J^{\text{diag}} = J + J^{\cE} \in \cW^{\ell} \otimes \cE$ denote the diagonal Heisenberg element, which satisfies $$J^{\text{diag}}(z) J^{\text{diag}}(w) \sim \frac{3+2\ell}{3} (z-w)^{-2}.$$

We will be interested in the commutant $\text{Com}( J^{\text{diag}}, \cW^{\ell} \otimes \cE)$. As with our study of $\cC^{\ell}$, we begin by describing $(\cW^{\ell} \otimes \cE)^{U(1)}$ where the action of $U(1)$ is generated by the zero mode of $J^{\text{diag}}$. It is easy to see that $(\cW^{\ell} \otimes \cE)^{U(1)}$ has the following strong generating set $$\{T, J, J^{\cE}, U_{0,n}, \phi^+_{0,n}, \phi^-_{0,n}|\  n\geq 0\},$$ where $\phi^+_{0,n} = \ :b \partial^n G^+ :$ and $\phi^-_{0,n} = \ :c \partial^n G^- :$. 

Observe that $\phi^{\pm}_{0,1} =  \pm : \phi^{\pm}_{0,0}J^{\cE}:$, and $\phi^{\pm}_{0,n} = \pm : \phi^{\pm}_{0,n-1} J^{\cE}:$, for $n\geq 0,$ so the generators $\{\phi^{\pm}_{0,n}|\ n>0\}$ are not necessary. From the previous relations, we know that $\{U_{0,n}|\ n\geq 5\}$ are not necessary as long as $\ell \neq 0$ or $1/2$. In fact, we have the relation
$$U_{0,1} = -:\phi^+_{0,0} \phi^-_{0,0}:  - :J^{\cE} U_{0,0}: - \frac{3}{2}  :J J \partial J^{\cE}: + 
 \frac{3 + 2 \ell}{4}  :T \partial J^{\cE}:  + \frac{3-6\ell}{8} :(\partial J) \partial J^{\cE}: $$ $$ +
  \frac{1 - 2 \ell}{4}  :J \partial^2 J^{\cE}: + \frac{\ell - 2 \ell^2}{24}  \partial^3 J^{\cE} + \partial U_{0,0} +
 \frac{3}{2}  :J J J^{\cE} J^{\cE}: -
\frac{3-6\ell }{8} :(\partial J) J^{\cE} J^{\cE}: 
$$ $$- \frac{3-6\ell}{4} :J (\partial J^{\cE})J^{\cE}: -
 \frac{\ell - 2 \ell^2}{24}  :JJJJ: + \frac{\ell-2\ell^2}{4} :(\partial J^{\cE}) J^{\cE} J^{\cE}: - 
 \frac{\ell - 2 \ell^2}{6}  :(\partial^2 J^{\cE}) J^{\cE}: $$ $$- 
 \frac{\ell - 2 \ell^2}{8} :(\partial J^{\cE}) \partial J^{\cE}:  + \frac{1-2\ell}{4}  :J J^{\cE} J^{\cE} J^{\cE}:  - \frac{3+2\ell}{4} :T J^{\cE} J^{\cE}:,$$ so for all $\ell \neq -3/2$, including $\ell = 0$ and $\ell = 1/2$, $U_{0,1}$ is not necessary. We have similar relations 
 $$U_{0,n+1} = (-1)^n :(\phi^+ (J^{\cE}( \cdots (J^{\cE} (\phi^-))\cdots ))):  - :J^{\cE} U_{0,n} + \cdots$$ for all $n\geq 1$, where the remaining terms depend only on $J, T, J^{\cE}$ and their derivatives, and have coefficients given by polynomials in $\ell$. This proves
 
\begin{lemma} \label{super-U(1)} For all $\ell \neq -3/2$, $(\cW^{\ell} \otimes \cE)^{U(1)}$ has a minimal strong generating set $$\{\phi^{\pm}_{0,0}, U_{0,0}, T,J,J^{\cE}\}.$$
\end{lemma}

Let $\cD^{\ell}$ denote the commutant $\text{Com}( J^{\text{diag}}, \cW^{\ell} \otimes \cE)$. Note that $$(\cW^{\ell} \otimes \cE)^{U(1)} \cong \cH \otimes \cD^{\ell},$$ where $\cH$ is generated by $J^{\text{diag}}$. Also, $\cD^{\ell}$ has a Virasoro element 
$$T^{\cD} = T+ T^{\cE} - \frac{3}{2(3+2\ell)} : J^{\text{diag}} J^{\text{diag}}:$$ and a Heisenberg field $$J^{\cD} = J - \frac{2\ell}{3} J^{\cE}.$$ The odd weight $2$ elements $\phi^{\pm} = \phi^{\pm}_{0,0}$, as well as the even weight $3$ element $U^{\cD} = U^{\cC}_{0,0}$, are easily seen to lie in $\cD^{\ell}$ as well. As in the previous section, Lemma \ref{super-U(1)} implies the following result.
\begin{thm} $\cD^{\ell}$ has a minimal strong generating set 
$$\{J^{\cD}, T^{\cD}, \phi^{\pm}, U^{\cD}\}.$$ \end{thm}
The full OPE algebra of $\cD^{\ell}$ is easy to compute from the explicit formulas for the generators, and it would be interesting to relate it to a super $\cW$-algebra associated to $\gs\gl(3|1)$. 

Next, we consider the simple vertex algebra $\cW_{\ell} \otimes \cE$, and we let $\cD_{\ell}$ denote the commutant $\text{Com}(J^{\text{diag}}, \cW_{\ell} \otimes \cE)$. By Corollary 2 of \cite{MiI}, $\cD_{\ell}$ is $C_2$-cofinite for all $\ell\in\Z_{>0}$.

\begin{thm}\label{thm:rationalsuper}
$\cD_\ell$ is rational for all positive integers $\ell$.
\end{thm}
\begin{proof}
$\cD_\ell$ is an extension of  $\cC_\ell\otimes V_{\sqrt{2\ell(2\ell+3)}\mathbb Z}$. The lattice vertex algebra is rational and for $\ell=1, 2$ we also know that $\cC_\ell$ is rational. The discriminant $\mathbb Z/2\ell(2\ell+3)\mathbb Z$ of the lattice $\sqrt{2\ell(2\ell+3)}\mathbb Z$ acts on $\cD_\ell$ as automorphism subgroup. The orbifold is $\cC_\ell\otimes V_{\sqrt{2\ell(2\ell+3)}\mathbb Z}$ and as a module for the orbifold
\[
\cD_\ell= \bigoplus_{n=0}^{2\ell(2\ell+3)-1} M_\ell,
\]
where each $M_\ell$ is a simple $\cC_\ell\otimes V_{\sqrt{2\ell(2\ell+3)}\mathbb Z}$-module \cite{DM}. It is in fact also $C_1$-cofinite as the orbifold is $C_2$-cofinite, hence Proposition 20 of \cite{MiII}
implies that each $M_\ell$ is a simple current. We thus have a simple current extension of a rational, $C_2$-cofinite vertex algebra of CFT-type and hence by \cite{Y} the extension is also rational.
\end{proof}

\begin{remark} $\cD_{\ell}$ is an example of a rational, $C_2$-cofinite vertex superalgebra whose conformal weight grading is by $\mathbb{N}$, not $\frac{1}{2}\mathbb{N}$.
\end{remark}

\section{Appendix}
Here we give the explicit relation of minimal weight $8$ among the generators of $(\cW^{\ell})^{U(1)}$.

$$ :U_{0,0} U_{1,1}: - :U_{0,1} U_{1,0}:  -\frac{3}{4}
 :(\partial^4 J) J J J: + 3 :(\partial^3 J)(\partial J) J J: - 
 \frac{9}{4} :(\partial^2 J)(\partial^2 J) J J:  $$ $$+ 18 :(\partial^2 J)(\partial J)(\partial J) J:  - 
 \frac{3}{32} (-1 + 2 \ell) :(\partial^5 J) J J:  - 
 \frac{3}{16} (-1 + 2 \ell) :(\partial^4 J)(\partial J) J: $$ $$- \frac{3}{4} (1 - 2 \ell) :(\partial^3 J) (\partial J) \partial J:  + 
 \frac{45}{16} (-1 + 2 \ell) :(\partial^2 J) (\partial^2 J) \partial J:  + 
 \frac{1}{16} (3 + 2 \ell) :(\partial^4 T) J J:  $$ $$- (\frac{3}{2} + \ell) :(\partial^3 T) (\partial J) J:  - (3/2 + \ell) :(\partial T)(\partial^3 J) J:  -  \frac{3}{2} (3 + 2 \ell) :(\partial T)(\partial^2 J) \partial J:  $$ $$+ 
  \frac{1}{8} (3 + 2 \ell) :T \partial^4 J) J:  +  \frac{1}{2} (3 + 2 \ell) :T \partial^3 J) \partial J: + 
  \frac{3}{8} (3 + 2 \ell) :T (\partial^2 J) \partial^2 J: $$ $$ - 
  \frac{27 - 36 \ell + 28 \ell^2 }{960} :(\partial^6 J) J:  - 
  \frac{93 - 84 \ell + 52 \ell^2}{640}  :(\partial^5 J) \partial J:  - 
  \frac{3 + \ell + 2 \ell^2 }{8} :(\partial^4 J) \partial^2 J:  $$ $$- 
  \frac{3 + \ell + 2 \ell^2}{12}  :(\partial^3 J) \partial^3 J:  - 
  \frac{3 - 4 \ell - 4 \ell^2}{160}  :(\partial^5 T) J: + 
  \frac{3 - 4 \ell - 4 \ell^2}{64}  :(\partial^4 T) \partial J: $$ $$ + 
  \frac{3 - 4 \ell - 4 \ell^2}{16}  :(\partial^3 T) \partial^2 J:  - 
  \frac{3 - 4 \ell - 4 \ell^2}{320}  :T \partial^5 J:  - 
  \frac{9 + 12 \ell + 4 \ell^2}{96}  :(\partial^4 T) T: $$ $$ + 
  \frac{9 + 12 \ell + 4 \ell^2}{24}  :(\partial^3 T) \partial T:  + \frac{3 - 13 \ell + 16 \ell^2 - 4 \ell^3}{2240} \partial^7J- \frac{3 \ell - 4 \ell^2 - 4 \ell^3}{1440} \partial^6 T$$ $$ - :(\partial^3 J) J U_{0,0}:  - 3 :(\partial^2 J)(\partial J) U_{0,0}: + 
 3 :(\partial^2 J) J U_{0,1}: + 3 :(\partial J)(\partial J) U_{0,1}: $$ $$ + 3 :(\partial J) J \partial^2 U_{0,0}:  - 
 6 :(\partial J) J \partial U_{0,1}:  + 3 :(\partial J) J U_{0,2}:  + :J J \partial^3 U_{0,0}: - 
 \frac{3}{2} :J J \partial^2 U_{0,1}: $$ $$+ \frac{3}{2} :J J \partial U_{0,2}: - 
 :J J U_{0,3}:  - \bigg(\frac{7}{8} + \frac{\ell}{4} \bigg) :(\partial^3 J) U_{0,1}: - 
 \frac{3}{8} (3 + 2 \ell) :(\partial^2 J) \partial^2 U_{0,0}:  $$ $$- \bigg( \frac{9}{4} + \frac{3 \ell}{2} \bigg) :(\partial^2 J)\partial U_{0,1}:  
 + \bigg(\frac{27}{8} + \frac{9 \ell}{4} \bigg) :(\partial^2 J) U_{0,2}:  - 
 \frac{3}{4} (3 + 2 \ell) :( \partial J) \partial^3 U_{0,0}: $$ $$+ \bigg(\frac{15}{8} + \frac{9 \ell}{4} \bigg) :(\partial J)\partial^2 U_{0,1}: + \bigg( \frac{3}{8} - \frac{3 \ell}{4} \bigg) :(\partial J) \partial U_{0,2}:  - 
 \frac{3}{8} (3 + 2 \ell) :J \partial^4 U_{0,0}:  $$ $$ + (2 + 2 \ell) :J \partial^3 U_{0,1}: - \bigg(\frac{3}{4} + \frac{3 \ell}{2} \bigg) :J \partial^2 U_{0,2}:  - \bigg(\frac{1}{4} - \frac{\ell}{2} \bigg) :J \partial U_{0,3}: $$ $$+ \bigg( \frac{1}{8} - \frac{\ell}{4} \bigg) :J U_{0,4}: + 
 \frac{1}{12} (3 + 2 \ell) :(\partial^3 T) U_{0,0}:  - \bigg( \frac{3}{4} + \frac{\ell}{2} \bigg) :(\partial^2 T) U_{0,1}: $$ $$- 
 \frac{1}{4} (3 + 2 \ell) :(\partial T) \partial^2  U_{0,0}: + \bigg(\frac{3}{2} + \ell\bigg) :(\partial T) \partial U_{0,1}: - \bigg(\frac{3}{4} + \frac{\ell}{2} \bigg) :(\partial T) U_{0,2}:  $$ $$- \frac{1}{6} (3 + 2 \ell) :T \partial^3 U_{0,0}:  + \bigg( \frac{3}{4} + \frac{\ell}{2} \bigg) :T \partial^2 U_{0,1}:  - \bigg( \frac{3}{4} + \frac{\ell}{2} \bigg) :T \partial U_{0,2}:  $$ $$+ \bigg(\frac{1}{2} + \frac{\ell}{3} \bigg) :T U_{0,3}:  + 
\bigg( \frac{45 + 58 \ell + 24 \ell^2}{120}\bigg)  \partial^5 U_{0,0} - \bigg(\frac{8 + 13 \ell+ 6 \ell^2}{8}\bigg) \partial^4 U_{0,1} $$ $$+ \bigg( \frac{7 + 14 \ell + 8 \ell^2}{8}\bigg) \partial^3 U_{0,2} - \bigg( \frac{3 + 7 \ell + 6 \ell^2}{12}\bigg)  \partial^2 U_{0,3} -\bigg(\frac{\ell (1 - 2 \ell)}{24}\bigg) \partial U_{0,4} +\bigg(\frac{\ell (1 - 2 \ell)}{60}\bigg) U_{0,5} .$$


\begin{thebibliography}{ABKS}

\bibitem[A-R]{A-R} H. Afshar, T. Creutzig, D. Grumiller, Y. Hikida, P. B. Ronne, \textit{Unitary W-algebras and three-dimensional higher spin gravities with spin one symmetry}, JHEP 1406 (2014) 063.
\bibitem[AM]{AM} D. Adamovic, A. Milas, \textit{On the triplet vertex algebra $\cW(p)$},  Adv. Math. 217 (2008), no. 6, 2664-2699.
\bibitem[ArI]{ArI} T. Arakawa, \textit{Representation theory of
	       superconformal algebras and the Kac-Roan-Wakimoto
	       conjecture}, Duke Math. J., Vol. 130, No. 3, pp. 435-478,
	       2005.
\bibitem[ArII]{ArII} T. Arakawa, \textit{Rationality of
	      Bershadsky-Polyalov vertex algebras}, Commun. Math. Phys.
	      323 (2013), no. 2, 627-633.
\bibitem[ArIII]{ArIII} T. Arakawa, \textit{ Representation theory of
	     W-algebras}. Invent. Math., 169(2):219--320, 2007.
\bibitem[ArIV]{ArIV} T. Arakawa, \textit{Associated varieties of modules
	      over Kac-Moody algebras and $C_2$-cofiniteness of
	      $\cW$-algebras}, Int. Math. Res. Not. (2015)  Vol.\ 2015 11605-11666.
\bibitem[ArV]{ArV} T. Arakawa, \textit{Rationality of $\cW$-algebras: principal nilpotent cases}, Ann. Math. 182(2):565--694, 2015.

\bibitem[ACKL]{ACKL} T. Arakawa, T. Creutzig, K. Kawasetsu, and A. Linshaw, \textit{Orbifolds and cosets of minimal $\cW$-algebras}, published online in Comm. Math. Phys. DOI 10.1007/s00220-017-2901-2.
\bibitem[AvE]{AvE} T. Arakawa, J. van Ekeren, \textit{Modularity of relatively rational vertex algebras and fusion rules of regular affine $\cW$-algebras}, arXiv:1612.09100.
\bibitem[ALY]{ALY} T. Arakawa, C.H. Lam and H. Yamada, \textit{Zhu's algebra, $C_2$-algebra and $C_2$-cofiniteness of parafermion vertex operator algebras}, Adv. Math. 264 (2014), 261-295.
\bibitem[ArM]{ArM} T. Arakawa, A. Moreau, \textit{Joseph ideals and lisse minimal W-algebras}, J. Inst. Math. Jussieu, published online.
\bibitem[Ber]{Ber} M. Bershadsky, \textit{Conformal field theories via Hamiltonian reduction},  139 (1991), no. 1, 71-82.
\bibitem[B-HI]{B-HI} R. Blumenhagen, W. Eholzer, A. Honecker, K. Hornfeck, and R. Hubel, {\it Unifying $\cW$-Algebras}, Phys. Lett. B 332 (1994).
\bibitem[B-HII]{B-HII} R. Blumenhagen, W. Eholzer, A. Honecker, K. Hornfeck, and R. Hubel, {\it Coset realizations of unifying $\mathcal{W}$-algebras}, Int. Jour. Mod. Phys. Lett. A10 (1995) 2367-2430.
\bibitem[Bor]{Bor} R. Borcherds, \textit{Vertex operator algebras, Kac-Moody algebras and the monster}, Proc. Nat. Acad. Sci. USA 83 (1986) 3068-3071.
\bibitem[C]{C} S. Carnahan, \textit{Building vertex algebras from parts}, arXiv:1408.5215.
\bibitem[CM]{CM} S. Carnahan and M. Miyamoto, \textit{Regularity of fixed-point vertex operator subalgebras}, arXiv:1603.05645.
\bibitem[CHR]{CHR} T. Creutzig, Y. Hikida and P. B. R\o nne, \textit{Higher spin AdS$_3$ supergravity and its dual CFT}, JHEP 1202 (2012) 109.
\bibitem[CKL]{CKL} T. Creutzig, S. Kanade, A. Linshaw, \textit{Simple current extensions beyond semi-simplicity}, arXiv:1511.08754.
\bibitem[CLI]{CLI} T. Creutzig, A. Linshaw, \textit{Cosets of affine vertex algebras inside larger structures}, arXiv:1407.8512v4.
\bibitem[CLII]{CLII} T. Creutzig, A. Linshaw, \textit{Orbifolds of symplectic fermion algebras}, Trans. Amer. Math. Soc. 369, no. 1 (2017), 467-494.
\bibitem[CLIII]{CLIII} T. Creutzig, A. Linshaw, \textit{The super $W_{1+\infty}$ algebra with integral central charge}, Trans. Amer. Math. Soc. 367, no. 8 (2015), 5521-5551.
\bibitem[CRW]{CRW} T. Creutzig, D. Ridout, S. Wood, \textit{Coset constructions of logarithmic (1, p) models}, Lett. Math. Phys. 104 (2014), no. 5, 553-583.
\bibitem[D]{D} C. Dong, \textit{Vertex algebras associated with even lattices}, J. Algebra 161 (1993), no. 1, 245--265.
\bibitem[DM]{DM} C. Dong, G. Mason \textit{Quantum Galois theory for compact Lie groups}, J. Algebra 214 (1999), no. 1, 92-102.
\bibitem[DLM]{DLM} C. Dong, H. Li, and G. Mason, \textit{Regularity of rational vertex operator algebras}, Adv. Math. 132 (1997), no. 1, 148-166.
\bibitem[DLY]{DLY} C. Dong, C.H. Lam and H. Yamada, \textit{$\cW$-algebras related to parafermion algebras}, J. Algebra 322 (2009), 2366-2403.
\bibitem[DLWY]{DLWY} C. Dong, C.H. Lam, Q. Wang and H. Yamada, \textit{The structure of parafermion vertex operator algebras},  J. Algebra 323 (2010), no. 2, 371-381.
\bibitem[DPYZ]{DPYZ} P. Di Vecchia, J. L. Petersen, M. Yu and H. B. Zheng, \textit{Explicit construction of unitary representations of the N = 2 superconformal algebra}, Phys. Lett. B 174 (1986) 280-284.
\bibitem[DR]{DR} C. Dong, L. Ren, \textit{Representations of the parafermion vertex operator algebras}, arXiv:1411.6085.
\bibitem[DWI]{DWI} C. Dong, and Q. Wang, \textit{The structure of parafermion vertex operator algebras: general case}, Comm. Math. Phys. 299 (2010), 783-792.
\bibitem[DWII]{DWII} C. Dong, and Q. Wang, \textit{On $C_2$-cofiniteness of the parafermion vertex operator algebras}, J. Algebra 328 (2011), 420-431.
\bibitem[DWIII]{DWIII} C. Dong, and Q. Wang, \textit{Parafermion vertex operator algebras}, Frontiers of Mathematics in China 6(4) (2011), 567-579.
\bibitem[FS]{FS} B. Feigin and A. Semikhatov, \textit{$\cW_n^{(2)}$ algebras}, Nuclear Phys. B 698 (2004), no. 3, 409-449.
\bibitem[FBZ]{FBZ} E. Frenkel and D. Ben-Zvi, \textit{Vertex Algebras and Algebraic Curves}, Math. Surveys and Monographs, Vol. 88, American Math. Soc., 2001. 
\bibitem[FKW]{FKW} E. Frenkel, V. Kac and M. Wakimoto, \textit{Characters and fusion rules for $\cW$-algebras via quantized Drinfeld-Sokolov reduction}, Comm. Math. Phys. 147 (1992), no. 2, 295--328.
\bibitem[FLM]{FLM} I. B. Frenkel, J. Lepowsky, and A. Meurman, \textit{Vertex Operator Algebras and the Monster}, Academic Press, New York, 1988.
\bibitem[FZ]{FZ} I. B. Frenkel, and Y.C. Zhu, \textit{Vertex operator algebras associated to representations of affine and Virasoro algebras}, Duke Mathematical Journal, Vol. 66, No. 1, (1992), 123-168.
\bibitem[GG]{GG} M. R. Gaberdiel and R. Gopakumar, \textit{An AdS$_3$ Dual for Minimal Model CFTs}, Phys. Rev. D 83 (2011) 066007.
\bibitem[GKO]{GKO} P. Goddard, A. Kent, and D. Olive, \textit{Virasoro algebras and coset space models},  Phys. Lett B 152 (1985) 88-93.
\bibitem[H\"o]{Ho} G. H\"ohn, \textit{Genera of vertex operator algebras and three-dimensional topological quantum field theories}, Fields Inst. Commun., 39, Amer. Math. Soc., Providence, RI, 2003.
\bibitem[HuI]{HuI} Y.-Z. Huang, \textit{Differential equations and intertwining operators}, Commun. Contemp. Math. 7 (2005), no. 3, 375-400. 
\bibitem[HuII]{HuII} Y.-Z.\ Huang, \textit{Vertex operator algebras and the Verlinde conjecture}, Commun.\ Contemp.\ Math.\ 10 (2008), 103-154.
\bibitem[Ka]{Ka} K. Kawasetsu, \textit{$\cW$-algebras with non-admissible levels and the Deligne exceptional series}, Int. Math. Res. Notices (2016), rnw240.
\bibitem[K]{K} V. Kac, \textit{Vertex Algebras for Beginners}, University Lecture Series, Vol. 10. American Math. Soc., 1998.
\bibitem[KP]{KP} V. Kac, D. Peterson, \textit{Infinite-dimensional Lie algebras, theta functions and modular forms}, Adv. Math. 53 (1984) 125-264.
\bibitem[KRW]{KRW} V. Kac, S.-S. Roan, M. Wakimoto, \textit{Quantum reduction for affine superalgebras}. Comm. Math. Phys. 241 (2003), no. 2-3, 307-342.
\bibitem[KWI]{KWI} V. Kac, M. Wakimoto M, \textit{Modular invariant representations of infinite-dimensional Lie algebras and superalgebras}. Proc. Nat. Acad. Sci. U.S.A., 85(14):4956-4960, 1988.
\bibitem[KWII]{KW} V. Kac, M. Wakimoto M, \textit{On rationality of W-algebras}, Transform. Groups 13 (2008), no. 3-4, 671-713.
\bibitem[LLY]{LLY} C. H. Lam, N. Lam and H. Yamauchi, \textit{Extension of unitary Virasoro vertex operator algebra by a simple module}, Int.\ Math.\ Res.\ Not.\ 2003, no.\ 11, 577-611. 
\bibitem[LiI]{LiI} H. Li, \textit{Local systems of vertex operators, vertex superalgebras and modules},
J. Pure Appl. Algebra 109 (1996), no. 2, 143-195.
\bibitem[LiII]{LiII} H. Li, \textit{Vertex algebras and vertex Poisson algebras}, Commun. Contemp. Math. 6 (2004) 61-110.
\bibitem[LL]{LL} B. Lian and A. Linshaw, \textit{Howe pairs in the theory of vertex algebras}, J. Algebra 317, 111-152 (2007).
\bibitem[LZ]{LZ} B. Lian and G. Zuckerman, \textit{Commutative quantum operator algebras}, J. Pure Appl. Algebra 100 (1995) no. 1-3, 117-139.
\bibitem[MiI]{MiI} M. Miyamoto, \textit{$C_2$-cofiniteness of cyclic-orbifold models}, Comm. Math. Phys. 335 (2015),   1279-1286.
\bibitem[MiII]{MiII} M. Miyamoto, \textit{Flatness and semi-rigidity of vertex operator algebras}, arXiv:1104.4675. 
\bibitem[Pol]{Pol} A. Polyakov, \textit{Gauge transformations and diffeomorphisms}, Internat. J. Modern Phys. A 5 (1990), no. 5, 833-842.
\bibitem[TW]{TW} A. Tsuchiya, S. Wood, \textit{The tensor structure on the representation category of the $W_{p}$ triplet algebra}, J. Phys. A 46 (2013) 445203.
\bibitem[Wa]{Wa} W. Wang, \textit{Rationality of Virasoro vertex algebras}, Int. Math. Res. Notices 1993, no. 7, 197-211.
\bibitem[Y]{Y} H. Yamauchi, \textit{Module categories of simple current extensions of vertex operator algebras}. J. Pure Appl. Algebra 189 (2004), no. 1-3, 315-328. 
\bibitem[CZh]{CZh} C.-J. Zhu, \textit{The BRST quantization of the nonlinear $WB_2$ and $W_4$ algebras}, Nucl. Phys. B418 (1994), 379-399.
\bibitem[YZh]{YZh}Y. Zhu, \textit{Modular invariants of characters of vertex operators}, J. Amer. Soc. 9 (1996) 237-302.





\end{thebibliography}
\end{document}